\documentclass{amsart}

\usepackage{amssymb,amscd,amsthm,amsxtra}
\usepackage{dsfont,latexsym}

\usepackage{mathrsfs}

\newtheorem{theorem}{Theorem}[section]
\newtheorem{corollary}[theorem]{Corollary}
\newtheorem{lemma}[theorem]{Lemma}
\newtheorem{prop}[theorem]{Proposition}

\theoremstyle{definition}

\theoremstyle{remark}
\newtheorem{rem}[theorem]{Remark}
\numberwithin{equation}{section}

\newcommand{\R}{{\mathds R}}
\newcommand{\N}{{\mathds N}}
\newcommand{\Z}{{\mathds Z}}

\renewcommand{\leq}{\leqslant}
\renewcommand{\le}{\leqslant}
\renewcommand{\geq}{\geqslant}
\renewcommand{\ge}{\geqslant}

\newcommand{\eps}{\varepsilon }
\renewcommand{\epsilon}{\varepsilon }

\newcommand{\PV}{{\rm{P.V.}}\,}

\def\Xint#1{\mathchoice
{\XXint\displaystyle\textstyle{#1}}%
{\XXint\textstyle\scriptstyle{#1}}%
{\XXint\scriptstyle\scriptscriptstyle{#1}}%
{\XXint\scriptscriptstyle\scriptscriptstyle{#1}}%
\!\int}
\def\XXint#1#2#3{{\setbox0=\hbox{$#1{#2#3}{\int}$ }
\vcenter{\hbox{$#2#3$ }}\kern-.6\wd0}}
\def\ave{\Xint-}

\newlength{\defbaselineskip}
\newcommand{\setlinespacing}[1]
           {\setlength{\baselineskip}{#1 \defbaselineskip}}

\author[S. Dipierro]{Serena Dipierro}
\address[Serena Dipierro]{School of Mathematics, University of Edinburgh, 
King's Buildings, Mayfield Road, Edinburgh EH9 3JZ, Scotland, UK
}
\email{serena.dipierro@ed.ac.uk}

\author[A. Figalli]{Alessio Figalli}
\address[Alessio Figalli]{Mathematics Department,
University of Texas at Austin,
RLM 8.100, 2515 Speedway Stop C1200,          
Austin TX 78712-1202, USA}  
\email{figalli@math.utexas.edu}

\author[E. Valdinoci]{Enrico Valdinoci}
\address[Enrico Valdinoci]{Weierstra{\ss} Institut f\"ur Angewandte Analysis und Stochastik, Hausvogteiplatz 11A, 10117 Berlin, Germany}
\email{enrico.valdinoci@wias-berlin.de}

\begin{document}



\thanks{{\it Acknowledgements}.
The first author has been supported by EPSRC grant  EP/K024566/1 
``Monotonicity formula methods for nonlinear PDEs''. 
The second author has been supported 
by NSF grant DMS-1262411 
``Regularity and stability results in variational problems''.
The third author has been supported by ERC grant 277749 ``EPSILON Elliptic
Pde's and Symmetry of Interfaces and Layers for Odd Nonlinearities''.}

\title[Dislocation dynamics in crystals]{Strongly nonlocal \\
dislocation dynamics in crystals}

\begin{abstract}
We consider the equation
$$ v_t=L_s v-W'(v)+\sigma_\epsilon(t,x) \quad {\mbox{ in }} (0,+\infty)\times\R,$$
where $L_s$ is an integro-differential operator of order~$2s$,
with~$s\in(0,1)$,
$W$ is a periodic potential, and~$\sigma_\epsilon$ is a small external stress.
The solution~$v$ represents the atomic dislocation
in the Peierls--Nabarro model 
for crystals, and we specifically consider the case~$s\in(0,1/2)$,
which takes into account a strongly nonlocal elastic term.

We study the evolution
of such dislocation function for macroscopic space and time scales,
namely we introduce the function
$$v_{\epsilon}(t,x):=v\left(\frac{t}{\epsilon^{1+2s}},\frac{x}{\epsilon}\right).$$
We show that, for small~$\epsilon$, the function~$v_\epsilon$ approaches the sum of step
functions. {F}rom the physical point of view, this shows that
the dislocations have the tendency to concentrate 
at single points of the crystal, where the size of the slip coincides with the
natural periodicity of the medium. We also show that the motion of
these dislocation points
is governed by an interior repulsive potential that is superposed
to an elastic reaction to
the external stress.
\end{abstract}

\maketitle

{\small
\tableofcontents 
}

\section{Introduction}

In this paper we deal with an integro-differential equation
of fractional order derived from the classical Peierls--Nabarro model
for crystal dislocations.
Specifically we will focus on the case in which
the fractional order of the equation is low, which corresponds
to a situation in which the long-range elastic interactions
give the highest contribute to the energy. In this framework,
we will describe the evolution of the atom dislocation
function by showing that, for sufficiently long times
and at a macroscopic scale, the dislocation function
approaches the superposition
of a finite number of dislocations. These individual
dislocations have size
equal to
the characteristic period of the crystal
and they occur at some
specific points, which in turn evolve according to a repulsive potential
and reacting elastically to the external stress.

More precisely,
we consider the problem
\begin{equation}\label{Pp}
v_t=L_s v-W'(v)+\sigma_\epsilon(t,x) \quad {\mbox{ in }} (0,+\infty)\times\R, \end{equation}
where~$s\in(0,1)$, $L_s$ is the so-called fractional Laplacian,
and~$W$ is a $1$-periodic potential.
More explicitly, given $\varphi\in C^2(\R)\cap L^\infty(\R)$ and $x\in\R$,
we define
$$ L_s \varphi(x):=\frac{1}{2} \int_\R \frac{\varphi(x+y)+\varphi(x-y)-2\varphi(x)}{|y|^{1+2s}}\,dy.$$
We refer to~\cite{SilvTH, DPV12} for a basic introduction 
to the fractional Laplace operator.
As for the potential, we assume that
\begin{equation}\label{Wass} 
\left\{ \begin{array}{lllll} 
W\in C^{3,\alpha}(\R), \quad {\mbox{ for some }} 0<\alpha<1, \\
W(x+1)=W(x) \quad {\mbox{ for any }}x\in\R, \\ 
W(k)=0 \quad {\mbox{ for any }}k\in\Z, \\  
W>0
\quad {\mbox{ in }}\R\setminus\Z, \\  
W''(0)>0. \end{array} \right.
\end{equation}
As customary,~$\epsilon>0$ is a small scale parameter, and $\sigma_{\epsilon}$
plays the role of
an exterior stress acting on the material. We suppose that
$$ \sigma_{\epsilon}(t,x):=\epsilon^{2s}\sigma(\epsilon^{1+2s}t,\epsilon x), $$
where~$\sigma$ is a bounded uniformly continuous function such that, 
for some~$\alpha\in(s,1)$ and~$M>0$, it holds 
\begin{equation}\begin{split}\label{sigma}
&\|\sigma_x\|_{L^{\infty}([0,+\infty)\times\R)}+\|\sigma_t\|_{L^{\infty}([0,+\infty)\times\R)} \leq M, \\
&|\sigma_x(t,x+h)-\sigma_x(t,x)|\leq M|h|^{\alpha}, \quad {\mbox{ for every }} x,h\in\R {\mbox{ and }} t\in[0,+\infty). 
\end{split}\end{equation}

The problem in \eqref{Pp}
arises in the classical
Peierls--Nabarro model for atomic
dislocation in crystals, see e.g.~\cite{Nab97}
and references therein. In this paper, our main focus
is on the fractional parameter range~$s\in(0,1/2)$,
which corresponds to a strongly nonlocal elastic term,
in which the energy contributions coming from far cannot
be neglected and, in fact, may become predominant. We refer to~\cite{GM12}
for the case~$s=1/2$ and to~\cite{DIPPV} for the case~$s\in(1/2,1)$.

We define
$$v_{\epsilon}(t,x):=v\left(\frac{t}{\epsilon^{1+2s}},\frac{x}{\epsilon}\right)$$
and we look at the equation satisfied by the rescaled function~$v_{\epsilon}$, that is, recalling \eqref{Pp},
\begin{eqnarray}\label{ACpar}
\left\{ 
\begin{array}{ll} 
(v_{\epsilon})_t=
\displaystyle\frac{1}{\epsilon}\biggl(L_s v_{\epsilon} - 
\displaystyle\frac{1}{\epsilon^{2s}}W'(v_{\epsilon})+\sigma(t,x)\biggr) \quad {\mbox{ in }}(0,+\infty)\times\R,\\
v_{\epsilon}(0,\cdot)=v_{\epsilon}^0 \quad {\mbox{ in }}\R. 
\end{array} 
\right.
\end{eqnarray} 
Following~\cite{PSV13, CSire}, we introduce the basic layer 
solution~$u\in C^{2,\alpha}(\R)$ (here~$\alpha=\alpha(s) \in(0,1)$), that is, the 
solution
of the problem
\begin{eqnarray}\label{AC}
\left\{ 
\begin{array}{ll} 
L_s u-W'(u)=0  \quad{\mbox{ in }}\R,\\
u'>0,\quad u(-\infty)=0, \quad u(0)=1/2, \quad u(+\infty)=1.
\end{array} 
\right.
\end{eqnarray} 
The name of layer solution is motivated by the fact that~$u$
approaches the limits~$0$ and~$1$ at~$\pm\infty$. More quantitatively,
there exists a constant~$C\geq1$ such that
\begin{equation}\label{phi}
|u(x)-H(x)|\leq C|x|^{-2s} \quad {\mbox{ and }} 
\quad |u'(x)|\leq C|x|^{-(1+2s)},
\end{equation}
where~$H$ is the Heaviside function, see Theorem~2 in~\cite{PSV13}.

As a preliminary result, we will prove a finer asymptotic
estimate on the decay of the layer solution:

\begin{theorem}\label{TH-decay}
Let $s \in (0,1/2)$.
There exist constants~$C>0$ and $\vartheta>2s$ such that
\begin{equation*}
\left|u(x)-H(x)+\frac{1}{2s\,W''(0)}\frac{x}{|x|^{1+2s}}\right|\leq\frac{C}{|x|^{\vartheta}
} \quad {\mbox{ for any }}x\in\R, \end{equation*}
with $\vartheta$ depending only on~$s$.
\end{theorem}

To state our next result, we recall that the semi-continuous envelopes of~$u$ 
are defined as 
$$ u^*(t,x):=\limsup_{(t',x')\rightarrow(t,x)}u(t',x') $$
and 
$$ u_*(t,x):=\liminf_{(t',x')\rightarrow(t,x)}u(t',x'). $$ 
Moreover, given~$x_1^0<x_2^0<\ldots<x_N^0$, 
we consider the solution~$\big(x_i(t)\big)_{i=1,\ldots,N}$ to the system 
\begin{eqnarray}\label{ODE2}
\left\{ 
\begin{array}{ll} 
\dot{x_i} =\gamma\biggl(-\sigma(t,x_i)+
\displaystyle\sum_{j\neq i}
\displaystyle\frac{x_i-x_j}{2s\,|x_i-x_j|^{2s+1}}\biggr) {\mbox{ in }}(0,+\infty), \\[4ex]
x_i(0)=x_i^0, \\
\end{array} 
\right.
\end{eqnarray} 
where
\begin{equation}\label{gamma}
\gamma=\left(\int_{\R}(u')^2\right)^{\!-1}.
\end{equation}
For the existence and uniqueness of such solution see Section~8 in~\cite{FIM09}.
We consider as initial 
condition in \eqref{ACpar} the state obtained by superposing $N$ copies
of the transition layers, centered at $x_1^0,\dots,x_N^0$, that is
\begin{equation}\label{initial} 
v_{\epsilon}^0(x)=\frac{\epsilon^{2s}}{\beta}\sigma(0,x)+ 
\sum_{i=1}^Nu\left(\frac{x-x_i^0}{\epsilon}\right), \end{equation} where 
\begin{equation}\label{beta} \beta:=W''(0)>0. \end{equation}
The main result obtained in this framework is the following:

\begin{theorem}\label{TH}
Let $s\in (0,1/2)$, assume that~\eqref{Wass}, \eqref{sigma} and~\eqref{initial} hold, 
and let 
\begin{equation*}
v_0(t,x)=\sum_{i=1}^N H(x-x_i(t)), \end{equation*}
where~$H$ is the Heaviside function and~$(x_i(t))_{i=1,\ldots,N}$ is the solution to~\eqref{ODE2}.

Then, for every~$\epsilon>0$ there exists a unique viscosity solution~$v_{\epsilon}$ 
to~\eqref{ACpar}. Furthermore, as~$\epsilon\rightarrow0$,
the solution~$v_{\epsilon}$ exhibits
the following asymptotic behavior:
\begin{equation*}
\limsup_{{(t',x')\rightarrow(t,x)}\atop{\epsilon\rightarrow0}}
v_{\epsilon}(t',x')\leq(v_0)^*(t,x) \end{equation*}
and 
\begin{equation*}
\liminf_{{(t',x')\rightarrow(t,x)}\atop{\epsilon\rightarrow0}}v_{\epsilon}(t',x')\geq(v_0)_*(t,x) 
\end{equation*}
for any~$t\in[0,+\infty)$ and~$x\in\R$. 
\end{theorem}

When $s=1/2$ the result 
above was proved in~\cite{GM12},
where it was also raised the question about
what happens for other values of the parameter~$s$.

In~\cite{DIPPV}, the result was extended to the case~$s\in(1/2,1)$.
So the main purpose of this paper was to obtain the result for
the remaining range of~$s\in(0,1/2)$. {F}rom the physical point of
view, this range of parameters is important since it corresponds to
the case of a strong nonlocal elastic effect: notice indeed that
the lower the value of~$s$ the stronger become the energy contributions
coming from far. We refer to~\cite{GM12, DIPPV} for a more exhaustive
set of physical motivations and heuristic asymptotics of the model we study.

We also remark that, differently from~\cite{GM12},
we do not make use of any harmonic extension results,
that are specific for the fractional powers of the Laplacian,
and so our proof is feasible for more general types of integro-differential
equations.

The cornerstone to prove Theorem~\ref{TH-decay}
(and hence Theorem~\ref{TH}) is given by the following
decay estimate at infinity, which we think has also independent interest:

\begin{theorem}\label{DECAY}
Let $s \in (0,1/2)$, and let~$v\in L^{\infty}(\R)\cap C^2(\R)$ such that 
\begin{equation} \label{va0}
\lim_{x\rightarrow\pm\infty}v(x)=0.\end{equation}
Suppose that there exists a function~$c\in L^{\infty}(\R)$ such that 
$c(x)\geq\delta>0$ for any~$x\in\R$ and for some~$\delta>0$, and 
\begin{equation}\label{eqM}
-L_s v +cv = g,
\end{equation}
where~$g$ is a function that satisfies the following estimate
\begin{equation}\label{Mest}
|g(x)|\leq\frac{C}{1+|x|^{4s}} \quad {\mbox{ for any }}x\in\R, 
\end{equation}
for some constant~$C\geq0$. 

Then, there exist $\vartheta\in(2s,1+2s]$ depending only on $s$, and
a constant~$\overline C\geq0$ depending on~$C$, $\delta$, $\|c\|_{L^\infty(\R)}$, and~$s$, 
such that 
$$ |v(x)|\leq\frac{\overline C}{1+|x|^{\vartheta}} \quad {\mbox{ for any }}x\in\R. $$
\end{theorem}

In our setting, we will use Theorem \ref{DECAY}
in the proof of
Theorem~\ref{TH-decay} (there, the function $v$
in the statement of Theorem \ref{DECAY}
will be embodied by the difference between
the solution~$u$ of problem \eqref{AC}
and a suitable heteroclinic solution of a model
problem, so that in this case condition \eqref{va0}
is automatically satisfied).

The explicit value of
the exponent~$\vartheta$ that appears
in the statement of Theorem \ref{DECAY}
will be given in formula~\eqref{SGAMMA}, but
such explicit value will not play any role in this paper (the only relevant
feature for us is that~$\vartheta>2s$). 
We think that it is an interesting open problem to determine
the optimal value of the exponent~$\vartheta$ in a general setting.

Theorem~\ref{DECAY}
may be seen as the strongly nonlocal version of Corollary~5.13
in~\cite{GM12} and
Corollary~7.1 in~\cite{DIPPV}, where similar decay estimates
(with different exponents) where obtained when~$s=1/2$ and~$s\in(1/2,1)$,
respectively. However, the techniques in~\cite{GM12, DIPPV}
are not sufficient to obtain the desired decay estimates
when~$s\in(0,1/2)$,
so the proof of Theorem~\ref{DECAY} here will rely on completely
different methods. Roughly speaking,
we use suitable test functions in order to obtain an integral
decay estimates (this will be accomplished in Proposition~\ref{PRTHM1})
and then we use barriers and sliding arguments to infer from it
a pointwise estimate. Remarkably, differently from the classical case
where pointwise estimates follow from integral ones using a suitable
version of the weak Harnack inequality (see e.g. Theorem~4.8(2)
in~\cite{CAFCAB}),
in our case, to the best of our knowledge, the fractional analog of this
weak Harnack inequality is not known. To overcome this difficulty,
some careful estimates on the fractional Laplacian of a function
below a barrier are employed (these estimates
will be obtained in Corollary~\ref{ULCO}).

The rest of the paper is organized as follows.
The proof of Theorem~\ref{DECAY} is contained in 
Sections~\ref{summation}--\ref{P:T}.
More precisely, we collect some preliminary elementary estimates
in Section~\ref{summation}. Then, in Sections~\ref{IEAP} and~\ref{ICAI},
we estimate the fractional Laplacian of a function
below a barrier
by taking into account the contribution in a neighborhood of
a given point and the contribution coming from infinity. An
integral decay estimate is given in Section~\ref{D:S}
and the proof of Theorem~\ref{DECAY} is completed in Section~\ref{P:T}.

With this we have the basic technical tools to prove 
Theorem~\ref{TH-decay} 
in Section~\ref{E:TH}. Then, Sections~\ref{L infty}--\ref{csiufff}
are devoted to the proof of Theorem~\ref{TH}. Namely,
Section~\ref{L infty} collects some uniform bounds that are
used in Section~\ref{7sddd} to construct the solution
of a corrector equation and prove its regularity.
With this, the proof of Theorem~\ref{TH} is
completed in Section~\ref{csiufff}.
 
\section{An auxiliary summation lemma}\label{summation}

Here we present some technical summation estimates, to be used in the
forthcoming Section~\ref{ICAI}. 
For the sake of generality, we prove the results in Sections \ref{summation}-\ref{D:S} in $\R^n$, for any $s \in (0,1)$ and $n\ge 1$. 

\begin{lemma}\label{pl1}
Let~$s \in (0,1)$, $x_0\in\R^n$ such that~$|x_0|\ge 3$, and $\vartheta\in(0,n+2s]$.
Then
$$ \sum_{{k\in \Z^n \setminus\{0\}}\atop{|x_0+k|\le |x_0|/2}}
\frac{1}{|k|^{n+2s}\,(1+|x_0+k|)^\vartheta}\le 
\frac{C}{(1+|x_0|)^\vartheta},$$
for some $C>0$ depending on $n$, $s$ and $\vartheta$. 
\end{lemma}

\begin{proof} If $|x_0+k|\le |x_0|/2$ then $|k|\ge |x_0|-|x_0+k|\ge
|x_0|/2$, therefore
\begin{equation}\label{9e1}
\sum_{{k\in \Z^n \setminus\{0\}}\atop{|x_0+k|\le |x_0|/2}}
\frac{1}{|k|^{n+2s}\,(1+|x_0+k|)^\vartheta}\le
\frac{2^{n+2s}}{|x_0|^{n+2s}}
\sum_{{k\in \Z^n \setminus\{0\}}\atop{|x_0+k|\le |x_0|/2}}
\frac{1}{(1+|x_0+k|)^\vartheta}.
\end{equation}
Moreover,
$$ \int_1^{|x_0|}
\frac{\rho^{n-1}\,d\rho}{\rho^\vartheta}= Z(n,\vartheta, x_0), $$ 
where 
$$ Z(n,\vartheta,x_0):=
\left\{ \begin{matrix}
(n-\vartheta)^{-1} (|x_0|^{n-\vartheta}-1) & {\mbox{ if }} n>\vartheta, 
\\
\log |x_0| & {\mbox{ if }} n=\vartheta,\\
(\vartheta-n)^{-1} (1-|x_0|^{n-\vartheta}) & {\mbox{ if }} n<\vartheta.
\end{matrix}
\right.$$
In any case
\begin{equation}\label{zeta}
\frac{Z(n,\vartheta,x_0)}{|x_0|^{n+2s}}\le\frac{c_{n,\vartheta}}{|x_0|^\vartheta},
\end{equation}
for some constant $c_{n,\vartheta}>0$ only depending on $n$ and $\vartheta$. Therefore
\begin{eqnarray*}
\sum_{{k\in \Z^n \setminus\{0\}}\atop{|x_0+k|\le |x_0|/2}}
\frac{1}{(1+|x_0+k|)^\vartheta} &\le& \int_{B_{|x_0|}(-x_0)} 
\frac{dx}{(1+|x+x_0|)^\vartheta} \\
&=& \omega_{n-1} \int_0^{|x_0|} 
\frac{\rho^{n-1}\,d\rho}{(1+\rho)^\vartheta}
\\ &\le& \omega_{n-1} \left[\int_0^{1} 
\rho^{n-1}\,d\rho
+\int_1^{|x_0|} 
\frac{\rho^{n-1}\,d\rho}{\rho^\vartheta}\right]
\\ &=& \omega_{n-1} \left[\frac1n
+ Z(n,\vartheta,x_0)
\right].\end{eqnarray*}
This and~\eqref{9e1} give that
$$
\sum_{{k\in \Z^n \setminus\{0\}}\atop{|x_0+k|\le |x_0|/2}}
\frac{1}{|k|^{n+2s}\,(1+|x_0+k|)^\vartheta} \le
\frac{C_1\left(1+Z(n,\vartheta,x_0)\right)}{|x_0|^{\vartheta}}, 
$$
for some $C_1>0$. Then, the desired result follows from \eqref{zeta}.
\end{proof}

\begin{corollary}\label{pc1.1}
Let~$s \in (0,1)$, $x_0\in\R^n$ such that~$|x_0|\ge 3$, and $\vartheta\in(0,n+2s]$.
Then
$$ \sum_{k\in \Z^n \setminus\{0\}}
\frac{1}{|k|^{n+2s}\,(1+|x_0+k|)^\vartheta}\le
\frac{C}{(1+|x_0|)^\vartheta},$$
for some $C>0$ depending on $n$, $s$ and $\vartheta$.
\end{corollary}

\begin{proof} Notice that
$$ \sum_{{k\in \Z^n \setminus\{0\}}\atop{|x_0+k|\ge |x_0|/2}}
\frac{1}{|k|^{n+2s}\,(1+|x_0+k|)^\vartheta}\le
\frac{1}{(1+|x_0|/2)^\vartheta}\sum_{k\in \Z^n 
\setminus\{0\}}
\frac{1}{|k|^{n+2s}}\le\frac{C_0}{(1+|x_0|)^\vartheta},$$
for some~$C_0>0$, and so the result follows from Lemma~\ref{pl1}.
\end{proof}

\section{Fractional Laplace computations I -- Integral estimates
at a point}\label{IEAP}

Here we estimate
the local contribution of the fractional Laplacian
of a function touched by above by a polynomial barrier.
By local, we mean here the contribution coming from a neighborhood 
of a given point. The contribution coming from far
will then be studied in Section~\ref{ICAI}.

Though the main focus of this paper is the fractional parameter
range~$s\in (0,1/2)$ the results presented hold true for any~$s\in(0,1)$.
For this, it is convenient to recall the notation
on singular integrals in the principal value sense, that is
$$ \PV
\int_{\R^n} \frac{u(x+y)-u(x)}{|y|^{n+2s}}\,dy
:= \lim_{\rho\searrow0}
\int_{\R^n\setminus B_\rho} \frac{u(x+y)-u(x)}{|y|^{n+2s}}\,dy.$$
As a matter of fact, when~$s\in(0,1/2)$ the above notation may be dropped
since the integrand is indeed Lebesgue summable
and no cancellations are needed to make the integral convergent
near the origin.

With this notation, we can estimate the
contribution in a given ball according to the following result:

\begin{lemma}\label{Con-L1}
Let~$s \in (0,1)$, $\vartheta>0$, $\epsilon\in(0,1)$, and
$$ F_1(x):=\frac{1}{(1+|x|)^\vartheta}.$$
For any fixed~$M>0$ let~$F_M (x):=M F_1(x)$.
Suppose that $u\in L^\infty(\R^n)\cap C^2(\R^n)$ satisfies
\begin{eqnarray}
&& F_M(x_0)+\epsilon=u(x_0) \ {\mbox{ for some point }} x_0\in\R^n, \label{S1}\\
&& F_M (x)+\epsilon\geq u(x) \ {\mbox{ for every }} x\in\R^n \label{S2}, \\
&& \int_{B_1(x_0)} |u(\zeta)|\,d\zeta \leq 
\frac{C_0}{(1+|x_0|)^{\vartheta}}        
\label{S3} \end{eqnarray}
for some~$C_0>0$.

Then there exists~$M_0>0$, depending only on~$n$, $s$, 
$\|u\|_{L^\infty(\R^n)}$, $\vartheta$, and~$C_0$,
such that if~$M\ge M_0$ then
$$ \PV\int_{B_1} \frac{u(x_0+y)-u(x_0)}{|y|^{n+2s}}\,dy \le 
-\frac{M\,|B_1|}{10\,(1+|x_0|)^{\vartheta}}.$$
\end{lemma} 

\begin{proof} First of all we observe that, without loss
of generality, we can suppose that
\begin{equation}\label{x0ma}
|x_0|>3.
\end{equation}
Indeed, if~$|x_0|\le 3$ we deduce from~\eqref{S1} that
$$ \frac{M}{4^\vartheta}\le
\frac{M}{(1+|x_0|)^\vartheta}=
F_M(x_0)=u(x_0)-\epsilon \le \|u\|_{L^\infty(\R^n)}$$
that gives an upper bound on~$M$ which would be
violated by choosing~$M_0$ large enough.

{F}rom~\eqref{x0ma}, we have that
\begin{equation}\label{yma}
{\mbox{for any $y\in B_1$, $|x_0+y|\ge |x_0|-|y|\ge |x_0|/2$.}}
\end{equation}
Now we define
\begin{eqnarray*}
D_1 &:=& \left\{ y\in B_1 {\mbox{ s.t. }} |u(x_0+y)| \ge
\frac{M}{2\,(1+|x_0|)^\vartheta} \right\},\\
D_2 &:=& \left\{ y\in B_1 {\mbox{ s.t. }} |u(x_0+y)| <
\frac{M}{2\,(1+|x_0|)^\vartheta} \right\}.
\end{eqnarray*}
Then, by~\eqref{S3},
$$ \frac{C_0}{(1+|x_0|)^{\vartheta}} \ge
\int_{D_1} |u(x_0+y)|\,dy \ge \frac{M\,|D_1|}{2\,(1+|x_0|)^\vartheta}.$$
Hence
\begin{equation}\label{mis D1}
|D_1| \le \frac{2C_0}{M}
\end{equation}
and, as a consequence,
if~$M$ is large enough, 
\begin{equation}\label{mis D2}
|D_2| \ge |B_1|-|D_1| \ge \frac{9\,|B_1|}{10}.
\end{equation}
Now we define
\begin{eqnarray*}
r_0 &:=& \left( \frac{(1+|x_0|)^2}{M}\right)^{1/(n+2)},
\\ D_3 &:=& D_1\cap B_{r_0}, \\
D_4 &:=& D_1\setminus B_{r_0}.
\end{eqnarray*}
If~$y\in D_3$ we use~\eqref{S1}, \eqref{S2}
and a Taylor expansion of~$F_1$ to obtain that
\begin{eqnarray*}
u(x_0+y)-u(x_0) &\le& M
\Big( F_1(x_0+y) -F_1(x_0) \Big)\\
&\le& 
M \nabla F_1(x_0)\cdot y+
M \sup_{\xi\in B_1} |D^2 F_1(x_0+\xi)|\,|y|^2.\end{eqnarray*}
Notice that
$$ |\partial^2_{x_i,x_j}F_1(x)| \le
\frac{2\vartheta}{(1+|x|)^{\vartheta+1}\,|x|}+
\frac{\vartheta\,(\vartheta+1)}{(1+|x|)^{\vartheta+2}}$$
and so, by~\eqref{x0ma} and~\eqref{yma},
$$ \sup_{\xi\in B_1} |D^2 F_1(x_0+\xi)| \le 
\frac{C_1}{(1+|x_0|)^{\vartheta+2}},$$
for some~$C_1>0$. Therefore, for any~$y\in D_3$,
$$ u(x_0+y)-u(x_0) \le M\nabla F_1(x_0)\cdot y+
\frac{C_1\, M\,|y|^2}{(1+|x_0|)^{\vartheta+2}}$$
and so, since the odd term vanishes in the principal value integral,
\begin{equation}\label{D3 est}
\begin{split}
\PV \int_{D_3} \frac{ u(x_0+y)-u(x_0) }{|y|^{n+2s}}\,dy\, &\le
\frac{C_1\, M}{(1+|x_0|)^{\vartheta+2}}
\int_{D_3} |y|^{2-n-2s}\,dy \\
&\le \frac{C_1\, M}{(1+|x_0|)^{\vartheta+2}}
\int_{B_{r_0}} |y|^{2-n-2s}\,dy \\
&= \frac{C_2\, M\, r_0^{2-2s}}{(1+|x_0|)^{\vartheta+2}}.
\end{split}
\end{equation}
Moreover, by~\eqref{S1}, \eqref{S2}, and~\eqref{yma},
we have that,
if~$y\in D_4$,
\begin{eqnarray*}
\frac{u(x_0+y)-u(x_0)}{|y|^{n+2s}} &\le& 
\frac{F_M(x_0+y)-F_M(x_0)}{|y|^{n+2s}} \\
&\le& \frac{F_M(x_0+y)}{|y|^{n+2s}} \\
&\le& \frac{M}{r_0^{n+2s} (1+|x_0+y|)^\vartheta}
\\ &\le& \frac{2^\vartheta\, M}{r_0^{n+2s} (1+|x_0|)^\vartheta}
.\end{eqnarray*}
Accordingly, making use of~\eqref{mis D1}, we conclude that
\begin{equation}\label{D4 est}
\begin{split}
\PV\int_{D_4} \frac{ u(x_0+y)-u(x_0) }{|y|^{n+2s}}\,dy\, &\le
\frac{2^\vartheta\, M\,|D_4|}{r_0^{n+2s} (1+|x_0|)^\vartheta} \\
&\le\frac{2^\vartheta\, M\,|D_1|}{r_0^{n+2s} (1+|x_0|)^\vartheta} \\
&\le\frac{C_3}{r_0^{n+2s} (1+|x_0|)^\vartheta}.
\end{split}
\end{equation}
for some $C_3>0$. Thus, by~\eqref{D3 est} and~\eqref{D4 est}, we obtain
\begin{equation}\label{D1 est}
\begin{split}
\PV\int_{D_1} \frac{ u(x_0+y)-u(x_0) }{|y|^{n+2s}}\,dy\, &\le
\frac{C_2\, M\, r_0^{2-2s}}{(1+|x_0|)^{\vartheta+2}}+
\frac{C_3}{r_0^{n+2s} (1+|x_0|)^\vartheta}\\
&\le\frac{C_4\, M^\beta}{(1+|x_0|)^{\vartheta+2\beta}}
\end{split}\end{equation}
for a suitable~$C_4>0$, where
\begin{equation}\label{beta def}
\beta:=\frac{n+2s}{n+2}\in (0,1).
\end{equation}
This completes the estimate of the contribution in~$D_1$.
Now we estimate the contribution in~$D_2$. For this, we notice that,
if~$y\in D_2$, then
$$ u(x_0+y)-u(x_0)=u(x_0+y)-\frac{M}{(1+|x_0|)^\vartheta}-\epsilon \le
-\frac{M}{2\,(1+|x_0|)^\vartheta}$$
and therefore
\begin{equation}\label{D2 est}
\begin{split}
\PV\int_{D_2} \frac{ u(x_0+y)-u(x_0) }{|y|^{n+2s}}\,dy\, &\le
-\frac{M}{2\,(1+|x_0|)^\vartheta}
\int_{D_2} \frac{ dy}{|y|^{n+2s}} \\
&\le -\frac{M}{2\,(1+|x_0|)^\vartheta}
\int_{D_2} \,dy\\
&\le -\frac{9 M\,|B_1|}{20\,(1+|x_0|)^\vartheta},\end{split}
\end{equation}
thanks to~\eqref{mis D2}.
By collecting the estimates in \eqref{D1 est}
and~\eqref{D2 est}, we obtain that
\begin{eqnarray*}
\PV\int_{B_1} \frac{ u(x_0+y)-u(x_0) }{|y|^{n+2s}}\,dy
&\le&\frac{C_4\, M^\beta}{(1+|x_0|)^{\vartheta+2\beta}}
-\frac{9 M\,|B_1|}{20\,(1+|x_0|)^\vartheta}
\\ &=& -\frac{9 M\,|B_1|}{20\,(1+|x_0|)^\vartheta} 
\left( 1 - \frac{C_5}{M^{1-\beta}\,(1+|x_0|)^{2\beta}} \right)
\\ &\le& -\frac{9 M\,|B_1|}{20\,(1+|x_0|)^\vartheta} 
\left( 1 - \frac{C_5}{M^{1-\beta}} \right)
\end{eqnarray*}
for some~$C_5>0$. So, since~$\beta\in(0,1)$ due to~\eqref{beta def},
for~$M$ large we obtain the desired result.
\end{proof}

\section{Fractional Laplace computations II -- Integral estimates
at infinity}\label{ICAI}

This is the counterpart of Section~\ref{IEAP}, since here we
study the contribution coming from infinity of the fractional Laplacian
of a function touched by above by a polynomial barrier
(since the singularity of the integral only occur at the origin,
we do not need to use the principal value notation for such contribution).

\begin{lemma}\label{Con-L2}
Let $s\in (0,1)$,~$\vartheta\in(0,n+2s]$, $\epsilon\in(0,1)$, and
$$ F_1(x):=\frac{1}{(1+|x|)^\vartheta}.$$
For any fixed~$M>0$ let~$F_M (x):=M F_1(x)$.
Suppose that $u\in L^\infty(\R^n)\cap C^2(\R^n)$ satisfies
\begin{eqnarray}
&& F_M(x_0)+\epsilon=u(x_0) \ {\mbox{ for some point }} x_0\in\R^n, \label{SS1}\\
&& F_M (x)+\epsilon\geq u(x) \ {\mbox{ for every }} x\in\R^n \label{SS2} \\
&& \int_{B_1(x)} |u(\zeta)|\,d\zeta \leq \frac{C_0}{(1+|x|)^{\vartheta}}
\ {\mbox{ for every }} x\in\R^n\label{S3S} \end{eqnarray}
for some~$C_0>0$.

Then there exists~$M_0>0$, depending only on~$n$, $s$,
$\|u\|_{L^\infty(\R^n)}$, $\vartheta$, and~$C_0$,
such that if~$M\ge M_0$ then
$$ \int_{\R^n\setminus B_1} \frac{u(x_0+y)-u(x_0)}{|y|^{n+2s}}\,dy \le
\frac{M\,|B_1|}{20\,(1+|x_0|)^{\vartheta}}.$$
\end{lemma}

\begin{proof} We notice that
$$ u(x_0+y)-u(x_0) =
u(x_0+y) - F_M(x_0)-\epsilon \le u(x_0+y)-\epsilon\le
\big( u(x_0+y)-\epsilon\big)^+.$$
Also, the cube centered at zero with side~$1/\sqrt{n}$ lies
inside the unit ball, namely~$Q_{1/\sqrt{n}}\subset B_1$.
Therefore
\begin{equation}\label{0o1} \int_{\R^n\setminus B_1} 
\frac{u(x_0+y)-u(x_0)}{|y|^{n+2s}}\,dy\le
\int_{\R^n\setminus Q_{1/\sqrt{n}}} 
\frac{\big(u(x_0+y)-\epsilon\big)^+
}{|y|^{n+2s}}\,dy.\end{equation}
Now we cover $\R^n\setminus Q_{1/\sqrt{n}}$ with cubes
of side~$1/(8n\sqrt{n})$ centered at points of a 
sublattice~${\mathcal{Z}}$ (roughly speaking, this sublattice
is just a scaling of~$\Z^n$ by a factor~$1/(8n\sqrt{n})$,
outside~$Q_{1/\sqrt{n}}$).
In this way,
\begin{equation}\label{0.9.5}
{\mbox{if $k\in {\mathcal{Z}}$,
then~$|k|\ge \displaystyle\frac{1}{2\sqrt{n}}$.}}\end{equation}
Therefore
\begin{equation}\label{09.1.1}
{\mbox{if $k\in {\mathcal{Z}}$ and $y\in Q_{1/(8n\sqrt{n})}(k)$
then~$|y|\ge |k|-|y-k|\ge \displaystyle\frac{|k|}{2}+
\displaystyle\frac{1}{4\sqrt{n}}-\displaystyle\frac{1}{8n}
\ge \frac{|k|}{2}$.}}\end{equation}
Moreover,
\begin{equation}\label{09.1.2}\begin{split}
&{\mbox{if $k\in {\mathcal{Z}}$ and $y\in Q_{1/(8n\sqrt{n})}(k)$
then}}\\ &1+|x_0+y|\ge 1+|x_0+k|-|y-k|\ge 1+|x_0+k|-
\displaystyle\frac{1}{8n}\ge\displaystyle\frac{1}{2}\big( 
1+|x_0+k|\big).\end{split}\end{equation}
Now we observe that, from~\eqref{0o1},
\begin{equation}\label{0.9.4}
\int_{\R^n\setminus B_1} \frac{u(x_0+y)-u(x_0)}{|y|^{n+2s}}\,dy \le
\sum_{k\in {\mathcal{Z}}}
\int_{Q_{1/(8n\sqrt{n})}(k)}
\frac{\big( u(x_0+y)-\epsilon\big)^+
}{|y|^{n+2s}}\,dy.\end{equation}
We define
\begin{eqnarray*}
D_1(k) &:=& \left\{ y\in Q_{1/(8n\sqrt{n})}(k)
{\mbox{ s.t. }} |u(x_0+y)| \ge
\frac{\sqrt{M}}{(1+|x_0+k|)^\vartheta} \right\}, \\
D_2(k) &:=& \left\{ y\in Q_{1/(8n\sqrt{n})}(k)
{\mbox{ s.t. }} |u(x_0+y)| <
\frac{\sqrt{M}}{(1+|x_0+k|)^\vartheta} \right\}.\end{eqnarray*}
Then, from~\eqref{S3S},
\begin{eqnarray*}
\frac{C_0}{(1+|x_0+k|)^{\vartheta}} &\ge&
\int_{B_1(x_0+k)} |u(\zeta)|\,d\zeta 
\\ &\ge& \int_{Q_{1/(8n\sqrt{n})}(x_0+k)} |u(\zeta)|\,d\zeta
\\ &\ge& \int_{D_1(k)} |u(x_0+y)|\,dy
\\ &\ge& \frac{\sqrt{M}\,|D_1(k)|}{(1+|x_0+k|)^\vartheta} 
\end{eqnarray*}
and so
$$ |D_1(k)|\le \frac{C_0}{\sqrt{M}}.$$
Consequently, using~\eqref{SS2}, \eqref{09.1.1}
and~\eqref{09.1.2}, we see that
\begin{equation}\label{0.9.0.11}
\begin{split}
\int_{D_1(k)} \frac{
\big( u(x_0+y)-\epsilon\big)^+
}{|y|^{n+2s}}\,dy\,&\le
\int_{D_1(k)} \frac{F_M(x_0+y)}{|y|^{n+2s}}\,dy \\
&\le \int_{D_1(k)} \frac{C_1\,M}{(1+|x_0+k|)^\vartheta\,|k|^{n+2s}}\,dy
\\ &=\frac{C_1\,M\,|D_1(k)|}{(1+|x_0+k|)^\vartheta\,|k|^{n+2s}}\\
&\le \frac{C_0\,C_1\,\sqrt{M}}{(1+|x_0+k|)^\vartheta\,|k|^{n+2s}},
\end{split}
\end{equation}
for a suitable~$C_1>0$. Now we use again~\eqref{09.1.1}
to estimate the 
contribution
in~$D_2(k)$ in the following computation:
\begin{equation}\label{0.9.0.12}
\begin{split}
\int_{D_2(k)} \frac{u^+(x_0+y)}{|y|^{n+2s}}\,dy\,&\le
\int_{D_2(k)} 
\frac{\sqrt{M}}{(|k|/2)^{n+2s}\,(1+|x_0+k|)^\vartheta} \,dy \\
&\le 
\frac{2^{n+2s}\,|Q_{1/(8n\sqrt{n})}|\,\sqrt{M}}{|k|^{n+2s}\,(1+|x_0+k|)^\vartheta}.
\end{split}
\end{equation}
Using~\eqref{0.9.0.11}
and~\eqref{0.9.0.12}, and the fact that
$$ \big( u(x_0+y)-\epsilon\big)^+\le u^+(x_0+y),$$
we conclude that
$$ \int_{Q_{1/(8n\sqrt{n})}(k)} \frac{
\big( u(x_0+y)-\epsilon\big)^+
}{|y|^{n+2s}}\,dy
\le\frac{C_2\,\sqrt{M}}{(1+|x_0+k|)^\vartheta\,|k|^{n+2s}},$$
for a suitable~$C_2>0$. So we plug this estimate into~\eqref{0.9.4}
and we deduce that
$$\int_{\R^n\setminus B_1} \frac{u(x_0+y)-u(x_0)}{|y|^{n+2s}}\,dy \le
C_2\,\sqrt{M} \sum_{k\in {\mathcal{Z}}}
\frac{1}{(1+|x_0+k|)^\vartheta\,|k|^{n+2s}}.$$
Thus we estimate the latter series using Corollary~\ref{pc1.1}
(notice that~${\mathcal{Z}}$ may be seen as a scaled version
of~$\Z^n\setminus\{0\}$, due to~\eqref{0.9.5},
and $x_0$ stays away from $0$,
as pointed out in \eqref{x0ma}, so the assumptions
of Corollary~\ref{pc1.1} are satisfied, up to scaling): we obtain that
$$\int_{\R^n\setminus B_1} \frac{u(x_0+y)-u(x_0)}{|y|^{n+2s}}\,dy \le
\frac{C_3\,\sqrt{M}}{(1+|x_0|)^\vartheta},$$
for a suitable~$C_3>0$, hence the claim plainly follows if~$M$
is large enough.
\end{proof}

Combining the estimates of
Lemmata~\ref{Con-L1} and~\ref{Con-L2}
we obtain that the negative local contribution
cannot be compensated by the contribution at infinity.
More explicitly, we have:

\begin{corollary}\label{ULCO}
Let $s\in (0,1)$, $\vartheta\in(0,n+2s]$, $\epsilon\in(0,1)$, and
$$ F_1(x):=\frac{1}{(1+|x|)^\vartheta}.$$
For any fixed~$M>0$ let~$F_M (x):=M F_1(x)$.
Suppose that $u\in L^\infty(\R^n)\cap C^2(\R^n)$ satisfies
\begin{eqnarray*}
&& F_M(x_0)+\epsilon=u(x_0) \ {\mbox{ for some point }} x_0\in\R^n, \\
&& F_M (x)+\epsilon\geq u(x) \ {\mbox{ for every }} x\in\R^n \\
&& \int_{B_1(x)} |u(\zeta)|\,d\zeta \leq \frac{C_0}{(1+|x|)^{\vartheta}}
\ {\mbox{ for every }} x\in\R^n\end{eqnarray*}
for some~$C_0>0$.

Then there exists~$M_0>0$, depending only on~$n$, $s$,
$\|u\|_{L^\infty(\R^n)}$, $\vartheta$, and~$C_0$,
such that if~$M\ge M_0$ then
\begin{equation}\label{ntg}
L_s u(x_0)=\PV\int_{\R^n} 
\frac{u(x_0+y)-u(x_0)}{|y|^{n+2s}}\,dy \le
-\frac{M\,|B_1|}{20\,(1+|x_0|)^{\vartheta}}.\end{equation}
\end{corollary}

\section{Decay estimates in average}\label{D:S}

Here we obtain some precise information
on the decay at infinity of the solution of
a nonlocal equation with decaying nonlinearity:
 
\begin{prop}\label{PRTHM1}
Let $s\in (0,1)$, $u\in L^\infty(\R^n)\cap C^2(\R^n)$ satisfy
\begin{equation}\label{122}
-L_s u + c u=g \quad {\mbox{ in }}\R^n, 
\end{equation}
where~$c(x)\in (c_0,\,c_0^{-1})$, for some~$c_0\in(0,1)$ 
and 
\begin{equation}\label{g} |g(x)|\le\frac{C}{(1+|x|)^{\alpha}} 
\end{equation}
for some~$C>0$ and $\alpha>0$. 

Then, for any $x\in\R^n$,
\begin{equation}\label{ball}
\int_{B_1(x)} |u(y)|\, dy
\le \frac{C_*}{|x|^{\vartheta}}
\end{equation}
where~$C_*>0$ is a suitable constant and
\begin{equation}\label{SGAMMA}
\vartheta := \frac{ \min\{ n+2s-(n-2\alpha)^+, \, 2\alpha\} }{2}.
\end{equation}
\end{prop}

\begin{proof}
We use that~$u$ satisfies~\eqref{122} in the weak sense,
that is, for any test function~$\psi$, 
$$ 
\int_{\R^n}\int_{\R^n}\frac{(u(x)-u(y))(\psi(x)-\psi(y))}{|x-y|^{n+2s}}\, 
dx\, dy + \int_{\R^n}c\,u\psi\, dx = \int_{\R^n}g\psi\, dx. $$ 
Choosing~$\psi=u\varphi^2$ we get
\begin{equation}\label{133}
\int_{\R^n}\int_{\R^n}\frac{\big(u(x)-u(y)\big)\,\Big(u(x)\varphi^2(x)-u(y)\varphi^2(y)\Big)}{|x-y|^{n+2s}}\, 
dx\, dy + \int_{\R^n}c\,u^2\varphi^2\, dx = \int_{\R^n}gu\varphi^2\, dx. 
\end{equation} 
Notice that we can write 
\begin{eqnarray*}
&&\big(u(x)-u(y)\big)\,\Big(u(x)\varphi^2(x)-u(y)\varphi^2(y)\Big) \\
&=& 
\big(u(x)-u(y)\big)\,\Big(u(x)\varphi^2(x)-u(y)\varphi^2(x)+u(y)\varphi^2(x)-u(y)\varphi^2(y)\Big) 
\\
&=&\big(u(x)-u(y)\big)\, \Big[\big(u(x)-u(y)\big)\varphi^2(x)+u(y)
\big(\varphi^2(x)-\varphi^2(y)\big)\Big] \\
&=& \big(u(x)-u(y)\big)^2\varphi^2(x)+u(y)\big(u(x)-u(y)\big)\,
\big(\varphi(x)+\varphi(y)\big)\,\big(\varphi(x)-\varphi(y)\big). 
\end{eqnarray*}
Hence~\eqref{133} becomes 
\begin{equation}\begin{split}\label{144}
&\int_{\R^n}\int_{\R^n}\frac{(u(x)-u(y))^2\varphi^2(x)}{|x-y|^{n+2s}}\, 
dx\, dy \\
&\quad +\int_{\R^n}\int_{\R^n}
\frac{u(y)\big(u(x)-u(y)\big)\,\big(\varphi(x)+\varphi(y)\big)
\,\big(\varphi(x)-\varphi(y)\big)}{|x-y|^{n+2s}}\, dx\, dy \\
&\quad +\int_{\R^n}c\,u^2\varphi^2\, dx = \int_{\R^n}gu\varphi^2\, dx. 
\end{split}\end{equation}
Now we estimate the second term in~\eqref{144} in the following way
\begin{eqnarray*}
&&\left|\int_{\R^n}\int_{\R^n}\frac{u(y)\big(u(x)-u(y)\big)
\,\big(\varphi(x)+\varphi(y)\big)\,\big(\varphi(x)-\varphi(y)\big)}{|x-y|^{n+2s}}\, 
dx\, dy\right| \\
&\le& 
\frac{1}{4}\int_{\R^n}\int_{\R^n}\frac{(u(x)-u(y))^2(\varphi(x)+\varphi(y))^2}{|x-y|^{n+2s}}\, 
dx\, dy + \int_{\R^n}\int_{\R^n} 
\frac{u^2(y)(\varphi(x)-\varphi(y))^2}{|x-y|^{n+2s}}\, dx\, dy\\
&\le& 
\frac12\int_{\R^n}\int_{\R^n}\frac{(u(x)-u(y))^2(\varphi^2(x)+\varphi^2(y))}{|x-y|^{n+2s}}\, 
dx\, dy +\int_{\R^n}\int_{\R^n} 
\frac{u^2(y)(\varphi(x)-\varphi(y))^2}{|x-y|^{n+2s}}\, dx\, dy\\
&\le& 
\int_{\R^n}\int_{\R^n}\frac{(u(x)-u(y))^2\varphi^2(x)}{|x-y|^{n+2s}}\, 
dx\, dy + \int_{\R^n}\int_{\R^n} 
\frac{u^2(y)(\varphi(x)-\varphi(y))^2}{|x-y|^{n+2s}}\, dx\, dy.
\end{eqnarray*}
Using this and \eqref{144} we obtain
\begin{equation}\begin{split}\label{145}
c_0 \int_{\R^n} u^2\varphi^2\, dx \, &\quad \le 
\int_{\R^n}c\,u^2\varphi^2\, dx\\
&\quad =\int_{\R^n}gu\varphi^2\, dx 
-\int_{\R^n}\int_{\R^n}\frac{(u(x)-u(y))^2\varphi^2(x)}{|x-y|^{n+2s}}\, 
dx\, dy\\
&\qquad\quad
-\int_{\R^n}\int_{\R^n}
\frac{u(y)(u(x)-u(y))(\varphi(x)+\varphi(y))
(\varphi(x)-\varphi(y))}{|x-y|^{n+2s}}\, dx\, dy\\
&\quad\le \int_{\R^n}gu\varphi^2\, dx + I,
\end{split}\end{equation}
where
$$ I:=\int_{\R^n}\int_{\R^n} 
\frac{u^2(y)(\varphi(x)-\varphi(y))^2}{|x-y|^{n+2s}}\, dx\, dy.$$
On the other hand
\begin{eqnarray*}
\int_{\R^n}gu\varphi^2\, dx&=&
\int_{\R^n} 2\,(\sqrt{1/(2c_0)} \,g\varphi) ( \sqrt{c_0/2}\, u\varphi)\, 
dx\\
&\le&
\frac{1}{2c_0}\int_{\R^n} g^2\varphi^2 +\frac{c_0}2
\int_{\R^n} u^2\varphi^2\, dx.\end{eqnarray*}
By plugging this into \eqref{145} and reabsorbing one term on the left 
hand side
we obtain
\begin{equation}\label{146}
\frac{c_0}2 \int_{\R^n} u^2\varphi^2\, dx 
\le\frac{1}{2c_0}\int_{\R^n} g^2\varphi^2\,dx + I.\end{equation}
Our goal is now twofold: to estimate $\int_{\R^n} g^2\varphi^2\,dx$ 
and
to reabsorb $I$ on the left hand side. For this, we choose 
$$ \varphi(x):=\frac{1}{(1+\eps^2|x-x_0|^2)^N}, $$ 
where $x_0\in\R^n$ is fixed, 
\begin{equation}\label{CH}
N:=\frac{n+2s}{4},
\end{equation}
and~$0<\eps\ll 1/N$. Notice that $\varphi\in L^2(\R^n)\cap 
L^\infty(\R^n)$. We set
\begin{equation}\label{390}
R:=|x_0|/2>10,
\end{equation}
and we claim that
\begin{equation}\label{Tw1}
\int_{\R^n} g^2\varphi^2\,dx\le C_\eps R^{-\gamma},
\end{equation}
for some $C_\eps>0$ and
$$ \gamma:=\min\{ n+2s-(n-2\alpha)^+, \, 2\alpha\}.$$
Notice that
\begin{equation}\label{out}
\vartheta=\gamma/2,
\end{equation}
see~\eqref{SGAMMA}.

To prove the claim, we first observe that if $x\in B_R$ then
$$ |x-x_0|\ge |x_0|-|x|\ge 2R-R=R,$$
so
$$ \varphi(x)\le\frac{1}{(1+\eps^2 R^2)^N}\le \frac{1}{\eps^{2N} 
R^{2N}}. $$
Accordingly, using also \eqref{g} and~\eqref{CH}, we obtain
\begin{equation}\label{Ty01}
\begin{split}
\int_{B_R} g^2\varphi^2\,dx &\quad\le\frac{1}{\eps^{4N} R^{4N}}
\int_{B_R} g^2 \,dx \\ &\quad\le \frac{1}{\eps^{4N} R^{4N}}\int_{B_R}
\frac{C}{(1+|x|)^{2\alpha}} \,dx\\ &\quad\le \frac{C}{\eps^{4N} 
R^{4N}}\left[ \int_{B_1}
1\,dx+\int_{B_R\setminus B_1}\frac{C}{|x|^{2\alpha}}\,dx\right]\\
&\quad \le C_\eps R^{-4N} \Big(1+\ell(R)\,R^{(n-2\alpha)^+}\Big)
\\
&\quad \le 2C_\eps \ell(R)\,R^{-n-2s+(n-2\alpha)^+},
\end{split}\end{equation}
for some $C_\eps>0$, where
$$ \ell(R):=\left\{
\begin{matrix}
\log R & {\mbox{ if }} 2\alpha=n,\\
1 & {\mbox{ otherwise.}}
\end{matrix}
\right.$$
Moreover, if $x\in B_R(x_0)$ then
$$ |x|\ge |x_0|-|x-x_0|\ge 2R-R=R$$
and so, from~\eqref{g}, we have
$$ |g(x)|\le\frac{C}{(1+R)^{\alpha}}\le\frac{C}{R^\alpha}.$$
As a consequence
\begin{equation}\label{Ty02}\begin{split}
\int_{B_R(x_0)} g^2\varphi^2\,dx\,&\quad\le\frac{C^2}{R^{2\alpha}}
\int_{B_R(x_0)} \varphi^2 \,dx \\&\quad\le 
\frac{C^2}{R^{2\alpha}} \int_{\R^n} \varphi^2 \,dx
\\&\quad\le 
C_\eps R^{-2\alpha},
\end{split}\end{equation}
for some $C_\eps>0$ (up to renaming it).
Now, if $x\in \R^n\setminus (B_R(x_0)\cup B_R)$ then $|x|\ge R$
and so, from~\eqref{g} and~\eqref{CH},
\begin{equation}\label{Ty03}\begin{split}
\int_{\R^n\setminus (B_R(x_0)\cup B_R)} 
g^2\varphi^2\,dx\,&\quad\le\frac{C^2}{R^{2\alpha}}
\int_{\R^n\setminus (B_R(x_0)\cup B_R)} \varphi^2 \,dx \\ &\quad
\le\frac{C^2}{R^{2\alpha}}
\int_{\R^n\setminus B_R(x_0)} \frac{1}{\eps^{4N} |x-x_0|^{4N}}\,dx 
\\ &\quad
\le C_\eps R^{n-2\alpha-4N}
\\ &\quad= C_\eps R^{-2\alpha-2s}.
\end{split}\end{equation}
Then \eqref{Tw1} follows from~\eqref{Ty01}, \eqref{Ty02} 
and~\eqref{Ty03}.

Now we claim that, for any $\eps'>0$, we can choose $\eps$ sufficiently small
(in the definition of $\varphi$) so that 
\begin{equation}\label{TBP}
\int_{\R^n}\frac{(\varphi(x)-\varphi(y))^2}{|x-y|^{n+2s}}\, dx\le 
\eps'\varphi^2(y),
\end{equation}
holds.

To prove this, we first observe that
\begin{equation}\label{nabla}
|\nabla\varphi(x)|=\frac{2\eps^2 N|x-x_0|}{(1+\eps^2|x-x_0|^2)^{N+1}}\le
2\eps N\varphi(x). \end{equation}
In particular we have that $|\nabla\varphi|\le 2\eps N$
and therefore, for any $r>0$,
\begin{eqnarray*}
\int_{\R^n}\frac{(\varphi(x)-\varphi(y))^2}{|x-y|^{n+2s}}\, dx &\le&
\int_{B_r(y)}\frac{4\eps^2 N^2 |x-y|^2}{|x-y|^{n+2s}}\, dx+
\int_{\R^n\setminus B_r(y)}\frac{4}{|x-y|^{n+2s}}\, dx \\ &\le&
C(\eps^2 r^{2-2s}+r^{-2s}),\end{eqnarray*}
for some $C>0$.
Accordingly, if we choose $r:=1/\sqrt{\eps}$, we obtain
$$ \int_{\R^n}\frac{(\varphi(x)-\varphi(y))^2}{|x-y|^{n+2s}}\, dx \le
2C\eps^s.$$ Hence if $y$ is such that $\eps |y-x_0|\le 
\eps^{-s/(4N)}/|\log\eps|$ then
we have that
\begin{eqnarray*}
|\log\eps|^{-N} \varphi^2(y) &=&
\frac{|\log\eps|^{-N}}{(1+\eps^2|y-x_0|^2)^{2N}} \\
&\ge&
\frac{|\log\eps|^{-N}}{\big(1+
(\eps^{-s/(4N)}/|\log\eps|)^2\big)^{2N}} \\
&\ge&
\frac{|\log\eps|^{-N}}{\big(2        
(\eps^{-s/(4N)}/|\log\eps|)^2\big)^{2N}}\\
&=&
2^{-2N} \eps^s |\log\eps|^{3N}
\\ &\ge& 2C\eps^s\\
&\ge& \int_{\R^n}\frac{(\varphi(x)-\varphi(y))^2}{|x-y|^{n+2s}}\, dx,
\end{eqnarray*}
provided that $\eps$ is small enough, and this shows that
\eqref{TBP} holds true if $\eps |y-x_0|\le \eps^{-s/(4N)}/|\log\eps|$.
So we may and do suppose that \begin{equation}\label{case}
\eps |y-x_0|\ge \eps^{-s/(4N)}/|\log\eps|.\end{equation}
Notice that, in this case, $\eps |y-x_0|\ge1$ if $\eps$ is small enough 
and so
\begin{equation}\label{QUI}
\varphi^2(y)=\frac{1}{(1+\eps^2|y-x_0|^2)^{2N}}\ge
\frac{1}{(2\eps^2|y-x_0|^2)^{2N}} =\frac{1}{4^N \eps^{n+2s} 
|y-x_0|^{n+2s}},
\end{equation}
thanks to \eqref{CH}.
Now we set
$$ r_\eps:=\frac{\eps^{-(n+3s)/(n+2s)}}{2|\log\eps|}$$
and we study the contributions in $B_{r_\eps}(x_0)$ and in 
$B_{r_\eps}(y)$.

For this, we point out that, by \eqref{CH} and \eqref{case},
\begin{equation}\label{Pe}
|y-x_0|\ge \frac{ \eps^{-(4N+s)/(4N)} }{|\log\eps|}=
\frac{ \eps^{-(n+3s)/(n+2s)} }{|\log\eps|}=2 r_\eps.
\end{equation} 
Therefore,
if $x\in B_{r_\eps}(x_0)$ we have that
$$ |x-y|\ge |x_0-y| - |x-x_0|\ge
|x_0-y| - r_\eps\ge\frac{|x_0-y|}2$$
hence, using \eqref{QUI}, we see that
\begin{equation}\label{AB1}
\begin{split}
\int_{B_{r_\eps}(x_0)}\frac{(\varphi(x)-\varphi(y))^2}{|x-y|^{n+2s}}\, 
dx \,&\le
\int_{B_{r_\eps}(x_0)}\frac{4^{n+1+2s}}{|x_0-y|^{n+2s}}\, dx\\
&\le C \frac{r_\eps^n}{|x_0-y|^{n+2s}} \\
&\le 4^N C 
\frac{\eps^{-n(n+3s)/(n+2s)}}{2|\log\eps|^n}\,\eps^{n+2s}\varphi^2(y) \\
&= 4^N C\frac{ \eps^{s(n+4s)/(n+2s)}}{2|\log\eps|^n}\, \varphi^2(y).
\end{split}
\end{equation}
Now we estimate the contribution in $B_{r_\eps}(y)$.
For this, we take $x\in B_{r_\eps}(y)$ and $\xi=tx+(1-t)y$ with 
$t\in[0,1]$
such that
$$ |\varphi(x)-\varphi(y)|\le |\nabla \varphi(\xi)|\,|x-y|.$$
Notice that, in this case,
$$ |\xi-y|=t|x-y|\le r_\eps \le\frac{|y-x_0|}2$$
thanks to \eqref{Pe}, and therefore
$$ |\xi-x_0|\ge |y-x_0|-|\xi-y|\ge\frac{|y-x_0|}2.$$
Using this and \eqref{nabla} we obtain that
\begin{eqnarray*}
|\nabla\varphi(\xi)|&\le& 2\eps N\varphi(\xi)
\\ &=& \frac{2\eps N}{(1+\eps^2|\xi-x_0|^2)^N}\\
&\le& \frac{2^{2N+1}\eps N}{(1+2^2\, \eps^2|\xi-x_0|^2)^N}\\
&\le& \frac{2^{2N+1}\eps N}{(1+\eps^2|y-x_0|^2)^N} \\
&=& 2^{2N+1}\eps N \varphi(y).
\end{eqnarray*}
As a consequence
\begin{equation}\label{AB2}
\begin{split}
\int_{B_{r_\eps}(y)}\frac{(\varphi(x)-\varphi(y))^2}{|x-y|^{n+2s}}\, dx 
\,&\le
\int_{B_{r_\eps}(y)}\frac{4^{2N+2}\eps^2 N^2 
\varphi^2(y)}{|x-y|^{n+2s-2}}\, dx\\
&= C \eps^2 r_\eps^{2-2s}\varphi^2(y)\\
&= \frac{C\,\eps^{2s(n-1+3s)/(n+2s)}}{2^{2-2s} \,|\log\eps|^{2-2s}} 
\varphi^2(y).
\end{split}
\end{equation}
It remains to estimate the contribution in $\R^n\setminus\big(
B_{r_\eps}(x_0)\cup B_{r_\eps}(y)\big)$.
For this we will use the following estimate:
fixed $p\in \R^n$ we have that
\begin{equation}\label{au6}
\int_{\R^n\setminus B_{r_\eps}(p)} 
\frac{dx}{|x-p|^{n+2s}}=\frac{C}{r_\eps^{2s}}=
2^{2s} C\,\eps^{2s(n+3s)/(n+2s)}\,|\log\eps|^{2s}
.\end{equation}
Moreover
$$ \frac{|y-x_0|}{|x-x_0|\,|x-y|}\le
\frac{|y-x|+|x-x_0|}{|x-x_0|\,|x-y|}=
\frac{1}{|x-x_0|}+\frac{1}{|x-y|}$$
and therefore
$$ \frac{|y-x_0|^{n+2s}}{|x-x_0|^{n+2s}\,|x-y|^{n+2s}}\le
2^{n+2s}\left(\frac{1}{|x-x_0|^{n+2s}}+\frac{1}{|x-y|^{n+2s}}\right).$$
Hence, if we integrate over $\R^n\setminus\big( 
B_{r_\eps}(x_0)\cup B_{r_\eps}(y)\big)$ and we use \eqref{au6}
we obtain that
\begin{equation}\label{au7}\begin{split}
&\int_{\R^n\setminus\big( 
B_{r_\eps}(x_0)\cup B_{r_\eps}(y)\big)}
\frac{|y-x_0|^{n+2s}}{|x-x_0|^{n+2s}\,|x-y|^{n+2s}}\,dx\\&\quad\le
2^{n+2s}\left(
\int_{\R^n\setminus B_{r_\eps}(x_0)}
\frac{dx}{|x-x_0|^{n+2s}}+
\int_{\R^n\setminus B_{r_\eps}(y)}
\frac{dx}{|x-y|^{n+2s}}\right)\\ &\quad\le C\,\eps^{2s(n+3s)/(n+2s)}\,
|\log\eps|^{2s},\end{split}\end{equation}
up to renaming constants. Moreover, exploiting \eqref{CH} and \eqref{QUI}
we see that
$$ \varphi^2(x)=\frac{1}{(1+\eps^2|x-x_0|^2)^{(n+2s)/2}} \le
\frac{1}{\eps^{n+2s} |x-x_0|^{n+2s}} \le \frac{4^N\,|y-x_0|^{n+2s}}{
|x-x_0|^{n+2s}}\,\varphi^2(y).$$
Therefore
\begin{equation}\label{AB3}
\begin{split}
&\int_{\R^n\setminus\big( 
B_{r_\eps}(x_0)\cup B_{r_\eps}(y)\big)}
\frac{\varphi^2(x)}{|x-y|^{n+2s}}\, dx \\&\quad\le
4^N \varphi^2(y)\,\int_{\R^n\setminus\big( 
B_{r_\eps}(x_0)\cup B_{r_\eps}(y)\big)}
\frac{|y-x_0|^{n+2s}}{
|x-x_0|^{n+2s}\,|x-y|^{n+2s}}\, dx \\
&\quad\le 4^N 
C\,\eps^{2s(n+3s)/(n+2s)}\,|\log\eps|^{2s}\,\varphi^2(y),\end{split}\end{equation}
thanks to \eqref{au7}. Furthermore, by \eqref{au6} we have that
\begin{equation}\label{AB4}
\begin{split}\int_{\R^n\setminus\big( 
B_{r_\eps}(x_0)\cup B_{r_\eps}(y)\big)}
\frac{\varphi^2(y)}{|x-y|^{n+2s}}\, dx \,&\le
\int_{\R^n\setminus
B_{r_\eps}(y)}
\frac{\varphi^2(y)}{|x-y|^{n+2s}}\, dx\\ &\le
2^{2s} 
C\,\eps^{2s(n+3s)/(n+2s)}\,|\log\eps|^{2s}\,\varphi^2(y).\end{split}\end{equation}
Now we use that
$$ (\varphi(x)-\varphi(y))^2 \le
(|\varphi(x)|+|\varphi(y)|)^2\le 
4(\varphi^2(x)+\varphi^2(y)),$$
so that
by \eqref{AB3} and \eqref{AB4} we obtain
\begin{equation}\label{22b} \int_{\R^n\setminus\big( 
B_{r_\eps}(x_0)\cup B_{r_\eps}(y)\big)}
\frac{(\varphi(x)-\varphi(y))^2}{|x-y|^{n+2s}}\, dx\le
C\,\eps^{2s(n+3s)/(n+2s)} \,|\log\eps|^{2s}\,\varphi^2(y),\end{equation}
up to renaming constants once again.
In view of~\eqref{AB1}, \eqref{AB2} and~\eqref{22b}, the proof
of \eqref{TBP} is finished. 

As a consequence of \eqref{TBP}
we obtain that
$$ I\le \eps' \int_{\R^n} u^2(y) \varphi^2(y)\, dy
=\eps' \int_{\R^n} u^2\varphi^2\,dx.$$
So we take $\eps$ so small that $\eps'\le c_0/4$,
we plug the estimate above into \eqref{146} and we reabsorb one
term into the left hand side (this fixes $\eps$ now once and for all): 
we conclude that
$$ \frac{c_0}4 \int_{\R^n} u^2\varphi^2\, dx
\le\frac{1}{2c_0}\int_{\R^n} g^2\varphi^2\,dx .$$
Hence, from~\eqref{Tw1},
$$ \frac{c_0}4 \int_{\R^n} u^2\varphi^2\, dx \le \frac{C_\eps}{2c_0} 
R^{-\gamma}.$$
Now we use that $\varphi\ge 1/2$ in $B_1(x_0)$ to deduce from this
that
$$ \ave_{B_1(x_0)} u^2\, dx
\le C R^{-\gamma},$$
for some $C>0$.
Then, by the H\"older inequality, \eqref{390}
and~\eqref{out}, for any~$x_0\in\R^n$
such that~$|x_0|>20$ we have that
$$ \ave_{B_1(x_0)} u\, dx \le \sqrt{
\ave_{B_1(x_0)} u^2\, dx}\le \sqrt{C 
R^{-\gamma}}=\sqrt{C}\,R^{-\vartheta}
=2^\vartheta\sqrt{C} |x_0|^{-\vartheta}.$$
Since~$u$ is bounded, a similar estimate holds for~$|x_0|\le20$
as well, by possibly changing the constants (also in dependence 
of~$\|u\|_{L^\infty(B_{20})}$).
This proves~\eqref{ball} and concludes the proof of 
Proposition~\ref{PRTHM1}.
\end{proof}

\begin{rem} {\rm In the sequel, we will only use
Proposition~\ref{PRTHM1} for the proof of Theorem~\ref{DECAY}
when~$n=1$ and~$s\in(0,1/2)$. Though the statement of
Proposition~\ref{PRTHM1} remains valid for
the whole parameter range~$s\in(0,1)$,
in general the exponent~$\vartheta$ found in~\eqref{SGAMMA}
would not be sufficiently accurate
(indeed, we think it is an interesting open
problem to find a sharp value for the exponent~$\vartheta$
in general).

The sensitivity of the decay estimates
on the fractional parameter~$s$ is the main reason for which different methods
are needed to prove Theorem~\ref{DECAY}
when~$s\in(0,1/2)$ and~$s\in[1/2,1)$: in a sense,
when~$s\in(0,1/2)$, the integral contributions
coming from far are predominant and they strongly
affect the available bounds on the asymptotic behaviour
of the solution at infinity.
}\end{rem}

\section{Proof of Theorem~\ref{DECAY}}\label{P:T}


Let~$v$ be as in Theorem~\ref{DECAY}.
We prove that
\begin{equation}\label{TY}
v(x)\le \frac{M_0}{(1+|x|)^\vartheta}
\end{equation}
for any~$x\in\R$, where~$M_0>0$ is a universal constant (the bound
from below follows by exchanging~$v$ with~$-v$).
To this goal, fixed any~$\eps>0$, we use~\eqref{va0}
to find~$R_\eps>0$ such that
\begin{equation}\label{Re}
{\mbox{$|v(x)|\le \eps/2$ for all~$|x|\ge R_\eps$.}}\end{equation}
We claim that
\begin{equation}\label{vb}
v(x)< \frac{M}{(1+|x|)^\vartheta}+\eps\end{equation}
for any~$x\in\R$, as long as
$$ M\ge \|v\|_{L^\infty(\R)}\,(1+R_\eps)^\vartheta.$$
To check this, we distinguish two cases. If~$|x|\le R_\eps$, then
$$ v(x)\le \frac{|v(x)|\,(1+R_\eps)^\vartheta }{(1+|x|)^\vartheta}
\le \frac{M}{(1+|x|)^\vartheta}<\frac{M}{(1+|x|)^\vartheta}+\eps,$$
proving~\eqref{vb} in this case. Conversely if~$|x|\ge R_\eps$, 
then~$v(x)<\eps$ and so~\eqref{vb} holds true in this case too.

Hence, we can take the smallest~$M:=M_\eps\ge 0$ for which~\eqref{vb}
is satisfied. If~$M_\eps=0$ for a sequence of~$\eps\searrow0$
then~\eqref{vb} gives that~$v(x)\le \eps$ and so, in the limit,~$v\le0$,
which proves~\eqref{TY}. Thus, without loss of generality, we can
suppose that~$M_\eps>0$. In this case, 
by \eqref{Re} and a simple compactness argument, 
there exists~$x_\eps\in\R$ for which
\begin{equation}\label{touch}
v(x_\eps)=\frac{M_\eps}{(1+|x_\eps|)^\vartheta}+\eps.\end{equation}
%
Our goal is to show that
\begin{equation}\label{ME}
M_\eps \le M_0
\end{equation}
for a suitable~$M_0>0$ independent of~$\eps$. For this,
we observe that,
by~\eqref{vb}, \eqref{touch} and Proposition~\ref{PRTHM1}
(with~$\alpha:=4s$), we have
that the hypotheses of Corollary~\ref{ULCO} are satisfied
(by taking~$u:=v$ and~$x_0:=x_\eps$). Therefore, by~\eqref{ntg},
if~$M_\eps$ were too large we would have that
\begin{equation}\label{p9} L_s v(x_\eps) \le
-\frac{M_\eps\,|B_1|}{20\,(1+|x_\eps|)^{\vartheta}}.\end{equation}
On the other hand, by~\eqref{touch}, \eqref{eqM},
and~\eqref{Mest}, we have
\begin{equation}\label{p8}\begin{split}
L_s v(x_\eps) \,&= L_s v(x_\eps) - c v(x_\eps) +c\left(
\frac{M_\eps}{(1+|x_\eps|)^\vartheta}+\eps\right)\\
&\ge L_s v(x_\eps) - c v(x_\eps)\\
&= -g(x_\eps)\\
&\geq -\frac{C}{(1+|x_\eps|)^{4s}} \\
&\geq -\frac{C}{(1+|x_\eps|)^{\vartheta}}
\end{split}
\end{equation}
(recall that~$\vartheta\le\alpha=4s$, see~\eqref{SGAMMA}).
Hence~\eqref{p8} and \eqref{p9} show that $M_\eps$ is universally bounded,
proving~\eqref{ME}.

{F}rom~\eqref{ME} we deduce that
$$ v(x)\le \frac{M_\eps}{(1+|x|)^\vartheta}+\eps\le
\frac{M_0}{(1+|x|)^\vartheta}+\eps$$
for any~$x\in\R$, and so, by letting~$\eps\searrow0$,
we obtain~\eqref{TY}.
This concludes\footnote{We remark that~$\vartheta$, as 
defined in~\eqref{SGAMMA}, satisfies
\begin{equation*}
\vartheta = \left\{\begin{array}{ll}
4s \quad
&{\mbox{ if }} s\in(0,1/6], \\
\frac{1+2s}{2} \quad &{\mbox{ if }}s\in(1/6,1/2).
\end{array} \right.
\end{equation*}
In any case, since~$s\in(0,1/2)$, we have that
$$ 2s<\vartheta<1+2s.$$}
the proof of
Theorem~\ref{DECAY}.

\section{Proof of Theorem~\ref{TH-decay}}\label{E:TH}

The proof of Theorem~\ref{TH-decay}
is now analogous to the one of Proposition~7.2
in~\cite{DIPPV}, up to the
following modifications, needed in the case~$s\in(0,1/2)$: 
\begin{itemize}
\item the exponent $1+2s$ in formulas~(7.9) and the previous one in \cite{DIPPV}
must be replaced by~$\vartheta$
(the rest of the argument remains unchanged, since~$\vartheta\in(2s,1+2s]$),
\item the use of Corollary~7.1 of~\cite{DIPPV} is replaced here
by Theorem~\ref{DECAY}.
\end{itemize}

\section{$L^\infty$ bounds}\label{L infty}

The goal of this section is to state some uniform regularity
estimates that will be needed in the subsequent
Section~\ref{7sddd}. 

We introduce the norm
\begin{equation}
\label{CLNC}
\|f\|_{{H}_0^{s}(\R^n)}
:= \sqrt{ \int_{\R^n}\int_{\R^n}
\frac{|f(x)-f(y)|^2}{|x-y|^{n+2s}}\, dx\, dy
}\end{equation}
and we provide an auxiliary estimate:

\begin{lemma}
Let $s \in (0,1)$. There exists a constant $C=C(n,s)>0$ such that, if
$f\in H^{s}(\R^n)$, then
\begin{equation}\label{L76}
\|f\|_{L^2(\R^n)} \le C\| f\|_{H^{s}_0(\R^n)}^{n/(n+2s)}\,
\|f\|_{L^1(\R^n)}^{2s/(n+2s)}.
\end{equation}

Also, if~$f\ge0$ then
\begin{equation}\label{L7678}
\|f\|_{L^2(\R^n)} \le C\| f\|_{H^{s}_0(\R^n)}\, |\{f>0\}|^{s/n}.
\end{equation}
\end{lemma}

\begin{proof}
We start by proving~\eqref{L76},
which is a variation of the classical
Nash inequality. Without loss of generality,
we suppose that~$f\in L^1(\R^n)$, otherwise the
right hand side of~\eqref{L76} is infinite and there is
nothing to prove. Given~$\rho>0$, we have
\begin{equation}\label{L73}
\int_{\R^n\setminus B_\rho} |\hat{f}(\xi)|^2\,d\xi
\le \int_{\R^n\setminus B_\rho} \frac{|\xi|^{2s}}{\rho^{2s}}
|\hat{f}(\xi)|^2\,d\xi\le C \rho^{-2s} \| f\|_{H^{s}_0(\R^n)}^2.\end{equation}
Here we have used the notation of the norm~$\|\cdot\|_{H^{s}_0(\R^n)}$,
as introduced in~\eqref{CLNC} and its equivalent in Fourier
spaces (see e.g. Proposition~3.4 in~\cite{DPV12}).
On the other hand,~$|\hat{f}(\xi)| \le \|f\|_{L^1(\R^n)}$ for
any $\xi\in\R^n$, and so by integrating over $B_\rho$ we obtain
$$ \int_{B_\rho} |\hat{f}(\xi)|^2\,d\xi \le |B_1|\,\rho^n \|f\|_{L^1(\R^n)}^2.$$
By adding this to~\eqref{L73} we obtain
$$ \|f\|_{L^2(\R^n)}^2 =\|\hat{f}\|_{L^2(\R^n)}^2
\le C \rho^{-2s} \| f\|_{H^{s}_0(\R^n)}^2 + |B_1|\,\rho^n \|f\|_{L^1(\R^n)}^2.$$
Since this estimate is valid for any~$\rho>0$, we now
choose
$$ \rho:= \big(\| f\|_{H^{s}_0(\R^n)} /\|f\|_{L^1(\R^n)}\big)^{2/(n+2s)} $$
to obtain 
\begin{equation*}
\|f\|_{L^2(\R^n)}^2 \le(C+|B_1|)\,\| f\|_{H^{s}_0(\R^n)}^{2n/(n+2s)}\,
\|f\|_{L^1(\R^n)}^{4s/(n+2s)}
,\end{equation*}
which gives~\eqref{L76}.

Now we prove~\eqref{L7678} by using~\eqref{L76}
and the H\"older inequality: we have
\begin{eqnarray*}
\|f\|_{L^2(\R^n)}^{n+2s} &\le& C\| f\|_{H^{s}_0(\R^n)}^n\,
\|f\|_{L^1(\R^n)}^{2s} \\
&\le& C\| f\|_{H^{s}_0(\R^n)}^n\,
\left[ \|f\|_{L^2(\R^n)}\,|\{f>0\}|^{1/2}
\right]^{2s}
\\ &=& C\| f\|_{H^{s}_0(\R^n)}^n\,\|f\|_{L^2(\R^n)}^{2s}\,
|\{f>0\}|^{s},
\end{eqnarray*}
which implies~\eqref{L7678}.
\end{proof}

We can now prove a uniform pointwise
estimate using a De Giorgi-type argument.
For the sake of generality, we prove it
for any~$s\in(0,1)$ and any~$n\ge 1$
(though we only need it here for~$n=1$ and~$s\in(0,1/2)$).

\begin{theorem}\label{thm_sup}
Let $s \in (0,1)$ and let~$\psi\in H^{s}(\R^n)$ be a weak solution to
\begin{equation*}
-L_s\psi = \lambda \psi +b \quad {\mbox{in }} \R^n, 
\end{equation*}
with $b,\lambda\in L^{\infty}(\R^n)$.
Then $\psi\in L^{\infty}(\R^n)$ and
\begin{equation*}
\|\psi\|_{L^{\infty}(\R^n)} \leq  C
\end{equation*}
where the constant $C>0$ depends only on $n$, $s$, 
$\|\psi\|_{L^2(\R^n)}$,
$\|\lambda\|_{L^{\infty}(\R^n)}$, and~$\| b\|_{L^{\infty}(\R^n)}$.
\end{theorem}

\begin{proof}
First, for any $0<\delta <<1$ (we will choose later a suitable $\delta$,
see formula \eqref{deltaC} below), we consider the function~$\phi$ defined as
\begin{equation*}
\phi(x):=\frac{\delta \psi(x)}{\|\psi\|_{L^2(\R^n)}}, \quad {\mbox{for any }} x\in\R^n.
\end{equation*}
By construction, 
\begin{equation*}
\|\phi\|_{L^2(\R^n)} = \delta,
\end{equation*}
and
\begin{equation}\label{eq_cresce}
-L_s\phi = \lambda \phi + \delta b/\|\psi\|_{L^2(\R^n)}.
\end{equation}
In order to prove the theorem, it will suffice to prove that
\begin{equation}\label{LtwoLinfty}
\| \phi\|_{L^{\infty}(\R^n)} \le 1, 
\end{equation}
since this implies that 
$$ \|\psi\|_{L^{\infty}(\R^n)}\le \frac{\|\psi\|_{L^2(\R^n)}}{\delta}\|\phi\|_{L^{\infty}(\R^n)} \le \frac{\|\psi\|_{L^2(\R^n)}}{\delta}
$$
and~$\delta$ is fixed.

Now, for any integer $k\in\N$, we consider the function $w_k$ defined as follows
$$
w_k(x):= (\phi(x)- (1 -2^{-k}))^+, \quad {\mbox{ for any }}x\in\R^n.
$$
By construction, $w_k\in H^{s}(\R^n)$, $w_k (\pm \infty) = 0$, and  
\begin{equation}\label{eq_9star}
w_{k+1}(x) \leq w_k(x) \quad {\mbox{a.e. in }} \R^n.
\end{equation}
The following inclusion
\begin{equation}\label{inclusion} 
\big\{ w_{k+1} > 0 \big\} \subseteq \big\{ w_k > 2^{-(k+1)} \big\}
\end{equation}
holds true for all $k\in\N$. Indeed, if $x\in\big\{ w_{k+1} > 0 \big\}$, then 
$$ 0<w_{k+1}(x)=\phi(x)-1+2^{-k-1}$$
hence
$$ \phi(x)-(1-2^{-k})> 2^{-k}-2^{-k-1}=2^{-k-1}$$
and so~$w_k(x)>2^{-k-1}$, thus proving~\eqref{inclusion}.
Moreover, we have the inequality
\begin{equation}\label{4.10bis}
\phi(x) < 2^{k+1} w_k(x) \quad {\mbox{ for any }} x \in  \big\{ w_{k+1} >0 \big\}.
\end{equation}
Indeed, if $x\in\big\{ w_{k+1} > 0 \big\}$ then 
$$ w_k(x)\ge w_{k+1}(x)=\phi(x)-(1-2^{-k-1}), $$ 
which together with \eqref{inclusion} implies 
\begin{eqnarray*}
\phi(x) &\le & w_k(x)+(1-2^{-k-1})=
w_k(x)+(2^{k+1}-1)2^{-k-1}\\ &<& w_k(x)+(2^{k+1}-1)w_k(x)=
2^{k+1}w_k(x). 
\end{eqnarray*}
This proves~\eqref{4.10bis}.

Also, we remark that for any $v\in H^{s}(\R^n)$ we have
\begin{equation}\label{uguale}
\big(v^+(x)-v^+(y)\big)\big(v(x)-v(y)\big)\geq | v^+(x) - v^+(y)|^{2},
\end{equation}
for all $x, y \in \R^n$.
In order to check this, let assume that $v(x) \geq v(y)$. There is no loss of generality in such assumption, since the roles of $x$ and $y$ can
be interchanged. Then, one can reduce to the case when 
$x\in\{ v> 0\}$ and $y\in\{ v \le 0 \}$, as otherwise the inequality in~\eqref{uguale}  plainly follows. Finally, we notice that in such a case \eqref{uguale} becomes
$$ (v(x)-v(y))v(x) \ge v(x)^{2} $$
which does hold since $v(y)\le0$ and $v(x)>0$.
This proves~\eqref{uguale}.

We now prove~\eqref{LtwoLinfty} by a standard iterative argument based on estimating the decay of the quantity
$$
U_k := \| w_k \|^2_{L^2(\R^n)}.
$$
First, in view of~\eqref{uguale} with $v:=\phi-(1-2^{-k})$, we have
\begin{eqnarray*}
\|w_{k+1}\|^2_{{H}_0^{s}(\R^n)} 
&:=& \int_{\R^n}\int_{\R^n} \frac{|w_{k+1}(x)-w_{k+1}(y)|^2}{|x-y|^{n+2s}}\, dx\, dy \\
&\leq& \int_{\R^n}\int_{\R^n}\frac{
\big(\phi(x)-\phi(y)\big)\big(w_{k+1}(x)-w_{k+1}(y)\big)}
{|x-y|^{n+2s}}\, dx\, dy.
\end{eqnarray*}
Thus, plugging $w_{k+1}$ as a test function in~\eqref{eq_cresce}, we obtain 
$$ \|w_{k+1}\|^2_{H_0^{s}(\R^n)}\leq  
\int_{\{w_{k+1}>0\}} \left(\lambda(x)\phi(x) + \frac{\delta\, b(x)}{\|\psi\|_{L^2(\R^n)}} \right)w_{k+1}(x)\,dx. $$
Notice that if $x\in\{w_{k+1}>0\}$ then $\phi(x)>0$, 
and therefore, using~\eqref{4.10bis} and \eqref{eq_9star}, we get
\begin{equation}\begin{split}\label{eq_9star2}
\|w_{k+1}\|^2_{H_0^{s}(\R^n)}\leq & \int_{\{w_{k+1}>0\}} 
\left( \sup_{\R^n}|\lambda|\, \phi(x)\,w_{k+1}(x)+
\frac{\delta \displaystyle\sup_{\R^n}|b|}{\|\psi\|_{L^2(\R^n)}}w_{k+1}(x)
\right)\, dx \\
\leq & \int_{\{w_{k+1}>0\}}
\left( \sup_{\R^n}|\lambda|\, 2^{k+1}\,w_k(x)\,w_{k+1}(x)+
\frac{\delta 
\displaystyle\sup_{\R^n}|b|}{\|\psi\|_{L^2(\R^n)}}w_{k+1}(x)
\right)\, dx \\
\leq & \int_{\{w_{k+1}>0\}}
\left( \sup_{\R^n}|\lambda|\, 2^{k+1}\,w_k^2(x)+
\frac{\delta
\displaystyle\sup_{\R^n}|b|}{\|\psi\|_{L^2(\R^n)}}w_{k}(x)
\right)\, dx \\
\leq & \sup_{\R^n}|\lambda|\, 2^{k+1}U_k + \frac{\delta\,\displaystyle\sup_{\R^n}|b|}{\|\psi\|_{L^2(\R^n)}}
\sqrt{|\{ w_{k+1}>0\}|} \; U_k^{\frac{1}{2}},
\end{split}\end{equation}
where we have also used the H\"older inequality.

Also, by~\eqref{inclusion} and Chebychev's inequality, one has
\begin{equation}\label{8.11bis}
|\{w_{k+1}>0\}| \le |\{ w_{k}>2^{-(k+1)}\}| \le  2^{2(k+1)} U_{k}, 
\end{equation}
so that~\eqref{eq_9star2} becomes
\begin{equation}\label{eq10ast2}
\|w_{k+1}\|^2_{H^s_0(\R^n)}\leq \left(\sup_{\R^n}|\lambda| + \frac{\delta \,\displaystyle\sup_{\R^n}|b|}{\|\psi\|_{L^2(\R^n)}} \right) 2^{k+1}U_k.
\end{equation}
On the other hand, using~\eqref{L7678}
(with~$f:=w_{k+1}$ here)
we have 
\begin{equation}\label{eq_10star}
\displaystyle
U_{k+1} \leq c \|w_{k+1}\|^2_{{H}_0^{s}(\R^n)} \big|\big\{ w_{k+1}>0\big\}\big|^{\frac{2s}{n}},
\end{equation}
where the constant $c>0$ only depends on $n$ and $s$.

Combining~\eqref{eq10ast2} with~\eqref{eq_10star} and using \eqref{8.11bis}, 
we get
\begin{equation*}\begin{split}
U_{k+1} &\leq\, c \left(\sup_{\R^n}|\lambda| + \frac{\delta\, \displaystyle\sup_{\R^n}|b|}{\|\psi\|_{L^2(\R^n)}} \right)2^{k+1}U_k \left(2^{2(k+1)}\right)^{\frac{2s}{n}}U_k^{\frac{2s}{n}}\\
&=\, c\left(\sup_{\R^n}|\lambda| + \frac{\delta \,\displaystyle\sup_{\R^n}|b|}{\|\psi\|_{L^2(\R^n)}} \right)
2^{(1+\frac{4s}{n})(k+1)} U_k^{1+\frac{2s}{n}} \\
&\le\, \left[1+c\left(\sup_{\R^n}|\lambda| + \frac{\delta\, \displaystyle\sup_{\R^n}|b|}{\|\psi\|_{L^2(\R^n)}} \right)\right]2^{(1+\frac{4s}{n})(k+1)}U_k^{1+\frac{2s}{n}}\\
&\le\, \left[
\left( 1+c \left(\sup_{\R^n}|\lambda| + \frac{\delta\, \displaystyle\sup_{\R^n}|b|}{\|\psi\|_{L^2(\R^n)}} \right)\right) 2^{1+\frac{4s}{n}} \right]^{k+1} 
U_k^{1+\frac{2s}{n}}\\
&=\, \bar{C}^{k+1} U_{k}^{1+\frac{2s}{n}}, 
\end{split}\end{equation*}
for some constant $\bar{C}>1$ depending on $\sup_{\R^n}|\lambda|$, $\sup_{\R^n}|b|$, 
$\|\psi\|_{L^2(\R^n)}$, $n$, and $s$.
Hence, an estimate of the form
$$ U_{k+1}\le \bar{C}^{k+1} U_{k}^{1+\alpha} \quad {\mbox{for any }} \ k\in\N,$$	 
holds for suitable $\bar{C}>1$ and $\alpha>0$. 

Now we perform our choice of~$\delta$, that is we assume that 
\begin{equation}\label{deltaC}
\delta^{2\alpha}=\frac{1}{\bar{C}^{(1/\alpha)+1}}. 
\end{equation}
We set 
\begin{equation}\label{deltaC bis}
\eta:=\frac{1}{\bar{C}^{1/\alpha}}. 
\end{equation}
Since $\bar C>1$ and $\alpha>0$, we have that 
\begin{equation}\label{starr}
\eta\in(0,1).
\end{equation}
We claim that
\begin{equation}\label{claim}
U_k \le\delta^2 \eta^k. 
\end{equation}
We show \eqref{claim} by induction. Indeed, we notice that 
$$ U_0:=\|w_0\|^2_{L^2(\R^n)}=\|\phi^+\|^2_{L^2(\R^n)}\le\|\phi\|^2_{L^2(\R^n)}=\delta^2, $$ 
which is \eqref{claim} for $k=0$. 
Now, suppose that \eqref{claim} is true for $k$ and let us prove it for $k+1$: 
$$ U_{k+1}\le \bar{C}^{k+1} U_k^{1+\alpha}\le 
\bar{C}^{k+1}(\delta^2\eta^k)^{1+\alpha} = \delta^2\eta^k (\bar{C}\eta^\alpha)^k \bar{C}\delta^{2\alpha}
= \delta^2\eta^{k+1},
$$
where we have used \eqref{deltaC} and \eqref{deltaC bis}. 
Then, by~\eqref{starr} and~\eqref{claim} we have that
\begin{equation}\label{limit}
\lim_{k\to\infty} U_{k} = 0.
\end{equation}

Noticing that 
$$ 0\le w_k=\left(\phi-(1-2^{-k})\right)^+\le |\phi|\in L^2(\R^n) $$ 
and 
$$ w_{k}\rightarrow(\phi-1)_{+}\quad {\mbox{a.e. in }}\R^{n}
\quad {\mbox{ as }}k\rightarrow +\infty,$$
by the Dominated Convergence Theorem we get
\begin{equation}\label{Ukappa}
\lim_{k\rightarrow +\infty}U_k=\|(\phi-1)^+\|_{L^2(\R^n)}^2. 
\end{equation}

Hence, from \eqref{limit} and \eqref{Ukappa} we have that $(\phi-1)^+=0$ 
almost everywhere in~$\R^n$, and so $\phi\le1$ almost everywhere in~$\R^n$. 
By replacing $\phi$ with $-\phi$ we get~\eqref{LtwoLinfty}, 
which concludes the proof. 
\end{proof}

\section{The corrector equation}\label{7sddd}

Now we consider the equation
\begin{eqnarray}\label{eq_correttore}
\left\{ 
\begin{array}{ll} 
L_s\psi-W''(u)\psi=u'+\eta\left(W''(u)-W''(0)\right) {\mbox{ in }}\R, \\
\psi\in H^s(\R), \\
\end{array} 
\right.
\end{eqnarray} 
where $u$ is the solution of \eqref{AC} and
\begin{equation}\label{eta2}
\eta=\frac{\displaystyle\int_{\R}(u'(x))^2\, dx}{W''(0)}. 
\end{equation} 
For a detailed heuristic motivation of such an equation
see Section~3.1 of~\cite{GM12}.

\begin{theorem}\label{THcorrettore}
There exists a unique solution~$\psi\in H^s(\R)$ to~\eqref{eq_correttore}. 
Furthermore
\begin{equation}\label{SCC}
{\mbox{$\psi\in C^{1,\alpha}_{loc}(\R)\cap L^{\infty}(\R)$
for some~$\alpha\in(0,1)$, and 
$\|\psi'\|_{L^{\infty}(\R)}<+\infty. $}}\end{equation}
\end{theorem}

\begin{proof} The proof is analogous to the one of Theorem~5.2
in~\cite{DIPPV}, where the result was
obtained for~$s\in(1/2,1)$, except for the modifications listed 
below.

The proof of Theorem~5.2
in~\cite{DIPPV} uses the condition~$s\in(1/2,1)$
only twice, namely before formula~(5.26)  and
at the end of Section~5. In the first occasion,
such condition was used to obtain that 
\begin{equation}\begin{split}\label{SCC2}
&{\mbox{a weak solution of $L_s v_0=W''(u) v_0$ is
$C^{2s+\alpha}(\R)\cap L^\infty(\R)$}} \\
&{\mbox{and, in particular, it is a classical solution.}}
\end{split}\end{equation}
In the second occasion, the condition on~$s$ was used to
obtain~\eqref{SCC}. In both the cases, the condition~$s\in(1/2,1)$
permitted to obtain the desired results as an easy consequence of
the fractional Morrey-Sobolev embedding (see e.g. Theorem 8.2
in~\cite{DPV12}), and this embedding is not available in the present
case.

Hence, we prove~\eqref{SCC} and~\eqref{SCC2} directly
from the regularity theory developed in Section~\ref{L infty},
thus obtaining that 
Theorem~\ref{THcorrettore} also holds when~$s\in(0,1/2)$.

To prove~\eqref{SCC2}, we first use Theorem~\ref{thm_sup} 
to obtain that~$v_0\in L^{\infty}(\R)$. Hence, from Proposition~5 in~\cite{SV13b} 
we deduce that~$v_0\in C^{\alpha}(\R)$ for any~$0<\alpha<2s$. 
In particular $v_0$ is a viscosity solution, and
since $W''(u)v_0\in C^{\alpha}(\R)$,
by Proposition~2.8 in~\cite{Sil06} 
we deduce that $v_0\in C^{\alpha+2s}(\R)$.
Thus $v_0$ is a classical solution, proving~\eqref{SCC2}.  
%

To show~\eqref{SCC}, we use Theorem~\ref{thm_sup} and Proposition~5 in~\cite{SV13b}
to obtain that~$\psi$ is a viscosity solution to~\eqref{eq_correttore} such that 
\begin{equation}\label{psi Linfinito}
\psi\in L^{\infty}(\R)\cap C^{\alpha}(\R) 
\end{equation} 
for any~$0<\alpha<2s$.

Now, we define the incremental quotient of~$\psi$ as 
$$ \psi_h(x):=\frac{\psi(x+h)-\psi(x)}{h} \quad {\mbox{for any }}x,h\in\R. $$
From~\eqref{eq_correttore} we have that~$\psi_h$ satisfies 
\begin{equation}\label{psi h}
L_s \psi_h(x)=W''(u(x+h))\psi_h(x) +W''_h(u(x))\psi(x)+ u'_h(x)+ \eta W''_h(u(x)) 
\end{equation}
where, for any~$x\in\R$, 
$$ u'_h(x):=\frac{u'(x+h)-u'(x)}{h} $$ 
and 
$$ W''_h(u(x)):=\frac{W''(u(x+h))-W''(u(x))}{h}. $$ 
From~\eqref{Wass}, \eqref{psi Linfinito}, and Lemma~6 in~\cite{PSV13}, we have that
$$ W''(u)\in L^\infty(\R) \quad {\mbox{ and}}\quad W''_h(u)\psi+u'_h+\eta W''_h(u)\in L^\infty(\R),$$ 
and so we can apply Theorem~\ref{thm_sup} to the solution of~\eqref{psi h}
to obtain that~$\psi_h\in L^{\infty}(\R)$. 
Using Proposition~5 in~\cite{SV13b}, this gives that
$\psi_h\in C^{\alpha}(\R)$ for any~$\alpha<2s$.

So we have proved that, for any~$x,y,h\in\R$, 
$$ |\psi_h(x)|\le C_1 \quad {\mbox{ and }}\quad |\psi_h(x)-\psi_h(y)|\le C_2|x-y|^\alpha, $$
for some positive constants~$C_1,C_2$. 
Letting~$h \searrow0$ we obtain that~$\psi'\in L^{\infty}(\R)\cap C^{\alpha}(\R)$, 
concluding the proof of~\eqref{SCC}. 
\end{proof}

\begin{rem}
Thanks to~\eqref{eq_correttore} and~\eqref{SCC}, 
we have that~$\psi\in H^s(\R)$ is uniformly continuous, 
and this implies that 
\begin{equation}\label{unif}
\displaystyle \lim_{x\rightarrow\pm\infty}\psi(x)=0. 
\end{equation}
\end{rem}

\section{Proof of Theorem~\ref{TH}}\label{csiufff}

The proof is now conceptually similar to the one given
in Section~8 of~\cite{DIPPV}, but some quantitative
estimates of Proposition~8.4 there need to be modified
when~$s\in(0,1/2)$.
For the facility of the reader, we provide
the details of the proof of Proposition~8.4 of~\cite{DIPPV}
in our case (this will be done in Proposition~\ref{y7} here below).

To this goal, we recall some of the notation of~\cite{GM12, DIPPV}
needed for our purposes.
We take an auxiliary parameter~$\delta>0$
and define~$(\overline x_i(t))_{i=1,\ldots,N}$ to
be the solution of the system 
\begin{eqnarray}\label{ODEdelta}
\left\{ 
\begin{array}{ll} 
\dot{\overline x}_i =\gamma\left(-\delta-\sigma(t,\overline x_i)+
\displaystyle\sum_{j\neq i}
\displaystyle\frac{\overline x_i-\overline x_j}{2s\, |\overline x_i-\overline x_j|^{1+2s}}\right) {\mbox{ in }}(0,+\infty), \\[4ex]
\overline x_i(0)=x_i^0-\delta.  \\
\end{array} 
\right.
\end{eqnarray} 
Moreover, we set
\begin{eqnarray}
&& \label{cbar} \overline c_i(t):=\dot{\overline x}_i(t) \\
&& \label{tildesigma}
\tilde\sigma:=\frac{\delta+\sigma}{\beta}, \quad {\mbox{ where }}\beta=W''(0)
{\mbox{ was introduced in }} \eqref{beta}, \\
&& \label{supsol}
\overline v_{\epsilon}(t,x):=
\epsilon^{2s}\tilde\sigma(t,x)+
\sum_{i=1}^N\left\lbrace u\left(
\frac{x-\overline x_i(t)}{\epsilon}\right)-
\epsilon^{2s}\overline c_i(t)
\psi\left(\frac{x-\overline x_i(t)}{\epsilon}\right)\right\rbrace, 
\end{eqnarray}
where~$u$ is given in Theorem~\ref{TH-decay} and~$\psi$ in Theorem~\ref{THcorrettore}. 
We set
\begin{equation}\label{utildei}
\tilde u_i:=u\left(\frac{x-\overline{x}_i(t)}{\epsilon}\right)- H\left(\frac{x-\overline{x}_i(t)}{\epsilon}\right), 
\end{equation}
where~$H$ is the Heaviside function,
$$ \psi_i:=\psi\left(\frac{x-\overline{x}_i(t)}{\epsilon}\right).  $$
and
\begin{equation}\label{Ieps}
I_{\epsilon}:=\epsilon(\overline{v}_\epsilon)_t+\frac{1}{\epsilon^{2s}}\left(W'(\overline{v}_\epsilon)-\epsilon^{2s}L_s\overline{v}_\epsilon-\epsilon^{2s}\sigma\right). 
\end{equation}
With this notation we have that (see Lemma~8.3
in~\cite{DIPPV}),
for every~$i_0\in\left\lbrace 1,\ldots,N\right\rbrace$,
\begin{equation}\label{10.6} 
I_{\epsilon}=e_{\epsilon}^{i_0}+(\beta\tilde\sigma-\sigma)+O(\tilde 
u_{i_0})\biggl(\eta\,  \overline c_{i_0}+\tilde\sigma+
\sum_{{1\le i\le N}\atop{i\neq i_0}}
\frac{\tilde 
u_{i}}{\epsilon^{2s}}\biggr), 
\end{equation}
where the error~$e_{\epsilon}^{i_0}$ is given by
\begin{equation}\label{ei}
e_{\epsilon}^{i_0}:=O(\epsilon^{2s})+\sum_{{1\le i\le N}\atop{i\neq 
i_0}}
O(\psi_i)+\sum_{{1\le i\le N}\atop{i\neq i_0}}O(\tilde u_i)+
\sum_{{1\le i\le N}\atop{i\neq i_0}}
O\left(\frac{\tilde u_i^2}{\epsilon^{2s}}\right). 
\end{equation}

Now we can state the following result, which replaces Proposition~8.4
in~\cite{DIPPV}:

\begin{prop}\label{y7}
There exists $\delta_0>0$ such that, for any $0<\delta\leq\delta_0$ 
and~$T>0$, 
we have 
$$ (\overline v_{\epsilon})_{t}\geq\frac{1}{\epsilon}\biggl(L_s\overline v_{\epsilon}- \frac{1}{\epsilon^{2s}}W'(\overline v_{\epsilon})+\sigma\biggr) \quad {\mbox{ in }} (0,T)\times\R, $$
for $\epsilon>0$ sufficiently small. 
\end{prop}

\begin{proof}
Recalling the definition of $I_{\epsilon}$ in \eqref{Ieps},
our goal is to show that
\begin{equation}\label{GP}
I_{\epsilon}\geq0\end{equation}
for $\epsilon$ small enough.
For this, we make a preliminary observation:
recalling the definition of 
$\tilde u_i$ in~\eqref{utildei} 
and using Theorem~\ref{TH-decay}, we obtain that, for 
any~$i\in\{1,\dots,N\}$,
\begin{equation}\label{999}
\left|\tilde u_i+\frac{\epsilon^{2s}}{2sW''(0)}\frac{x-\overline 
x_i(t)}{|x-\overline x_i(t)|^{1+2s}}\right|
\leq \frac{C\, \epsilon^{\vartheta}}{|x-\overline x_i(t)|^{\vartheta}}. 
\end{equation}
Since~$\vartheta>2s$, we can choose~$\gamma$ such that
\begin{equation}\label{scelta}
0<\gamma<\frac{\vartheta-2s}{\vartheta}. 
\end{equation}
Now we divide the proof of~\eqref{GP}
by dealing with two separate cases.
\noindent
\\{\it Case~1:} Suppose that there exists 
$i_0\in\left\lbrace1,\ldots,N\right\rbrace$ 
such that 
\begin{equation}\label{caso1}
|x-\overline x_{i_0}(t)|\leq\epsilon^{\gamma}.
\end{equation} 
Therefore, since the~$\overline x_{i}$'s are well-separated,
for $\epsilon$ sufficiently small we have that 
\begin{equation}\label{varteta}
|x-\overline x_{i}(t)|\geq\kappa>0, \ {\mbox{ for any }} \ i\ne i_0,
\end{equation}
where $\kappa$ is a constant independent of $\epsilon$. 

Hence, thanks to~\eqref{999} and~\eqref{varteta},
$$ \biggl|\sum_{i\neq i_0}\left(\frac{\tilde u_i}{\epsilon^{2s}}+\frac{1}{2sW''(0)}\frac{x-\overline x_i(t)}{|x-\overline x_i(t)|^{1+2s}}\right)\biggr|\leq \frac{C\, \epsilon^{\vartheta}}{\epsilon^{2s}}\sum_{i\neq i_0}\frac{1}{|x-\overline x_i(t)|^{\vartheta}}\leq C\, \epsilon^{\vartheta-2s}. $$
Therefore, from 
\eqref{10.6}, we deduce that 
\begin{equation}\begin{split}\label{666}
I_{\epsilon}=&\ e^{i_0}_{\epsilon}+\beta\tilde\sigma-\sigma+
O(\tilde u_{i_0})\biggl(\eta\,  \overline c_{i_0}+\tilde\sigma
+\sum_{i\neq i_0}\frac{\tilde u_{i}}{\epsilon^{2s}}\biggr)\\
=& \ e^{i_0}_{\epsilon}+\beta\tilde\sigma-\sigma
+O(\tilde u_{i_0})\biggl(\eta\, \overline c_{i_0}
+\tilde\sigma -\frac{1}{2sW''(0)}\sum_{i\neq i_0}
\frac{x-\overline x_i(t)}{|x-\overline x_i(t)|^{1+2s}} \biggr) + O(\epsilon^{\vartheta-2s}).
\end{split}\end{equation}

Now, we Taylor expand the function 
$\frac{x-\overline x_i(t)}{|x-\overline x_i(t)|^{1+2s}}$ for~$x$
in a neighborhood 
of the point~$\overline x_{i_0}(t)$, and we use~\eqref{caso1} to get 
\begin{equation}\begin{split}\label{8.17bis}
&\biggl|\sum_{i\neq i_0}\frac{x-\overline x_i(t)}{|x-\overline x_i(t)|^{1+2s}}- 
\sum_{i\neq i_0}\frac{\overline x_{i_0}(t)-\overline x_i(t)}{|\overline x_{i_0}(t)-\overline x_i(t)|^{1+2s}}\biggr|\\ 
&\qquad \qquad = \ \biggl|\sum_{i\neq i_0}\left(\frac{1}{|\xi-\overline x_i(t)|^{1+2s}} -(1+2s)\frac{(\xi-\overline x_i(t))^2}{|\xi-\overline x_i(t)|^{3+2s}}\right)(x-\overline x_{i_0}(t))\biggr| \\
& \qquad \qquad \leq\ \sum_{i\neq i_0}\frac{2+2s}{|\xi-\overline x_i(t)|^{1+2s}}\, \epsilon^{\gamma}\\
& \qquad \qquad \leq\ C\, \epsilon^{\gamma}, 
\end{split}\end{equation}
where $\xi$ is a suitable point lying on the segment joining $x$ to $\overline x_{i_0}(t)$ 
(and hence $|\xi-\overline x_i(t)|\geq\kappa/2$ thanks to~\eqref{caso1}). 
Therefore, using~\eqref{8.17bis} in~\eqref{666}, we have 
\begin{equation}\begin{split}\label{777}
I_{\epsilon}=&\ e^{i_0}_{\epsilon}+\beta\tilde\sigma-\sigma +O(\tilde u_{i_0})\biggl(\eta\, \overline c_{i_0}+\tilde\sigma -\frac{1}{2sW''(0)}\sum_{i\neq i_0}\frac{\overline x_{i_0}(t)-\overline x_i(t)}{|\overline x_{i_0}(t)-\overline x_i(t)|^{1+2s}} \biggr) \\
&\ + O(\epsilon^{\vartheta-2s}) +O(\epsilon^{\gamma}).
\end{split}\end{equation}
Now, we compute the term in parenthesis. From the definitions of 
$\eta$, $\overline c_{i_0}$ and $\tilde\sigma$ given in \eqref{eta2}, 
\eqref{cbar}, and \eqref{tildesigma} respectively, and 
recalling~\eqref{gamma},
we obtain
\begin{equation}\label{1114}\begin{split}
&\eta\, \overline c_{i_0}+\tilde\sigma -
\frac{1}{2sW''(0)}\sum_{i\neq i_0}\frac{\overline x_{i_0}(t)-
\overline x_i(t)}{|\overline x_{i_0}(t)-\overline x_i(t)|^{1+2s}} \\
=\; & \frac{1}{\gamma\, W''(0)} \dot{\overline x}_{i_0}(t)+ 
\frac{\delta}{W''(0)}+\frac{\sigma(t,x)}{W''(0)}
-\frac{1}{2sW''(0)}\sum_{i\neq i_0}\frac{\overline x_{i_0}(t)-\overline x_i(t)}{|\overline x_{i_0}(t)-\overline x_i(t)|^{1+2s}} \\
= \;& \frac{1}{W''(0)}\biggl(\frac{ 
\dot{\overline x}_{i_0}(t)
}{\gamma}+\delta+\sigma(t,\overline x_{i_0}(t))-\frac{1}{2s}\sum_{i\neq i_0}\frac{\overline x_{i_0}(t)-\overline x_i(t)}{|\overline x_{i_0}(t)-\overline x_i(t)|^{1+2s}}\biggr) \\ 
&\quad \ +\,\frac{\sigma(t,x)-\sigma(t,\overline x_{i_0}(t))}{W''(0)}. 
\end{split}\end{equation}
Recalling~\eqref{ODEdelta}, we have that 
$$ \frac{\dot{\overline x}_{i_0}(t)}{\gamma}+\delta+\sigma(t,\overline 
x_{i_0}(t))-\frac{1}{2s}\sum_{i\neq i_0}\frac{\overline x_{i_0}(t)-\overline x_i(t)}{|\overline x_{i_0}(t)-\overline x_i(t)|^{1+2s}}=0, $$
and so the term in parenthesis
in~\eqref{1114} vanishes. Therefore~\eqref{1114} becomes
\begin{eqnarray*}
\eta\, \overline c_{i_0}+\tilde\sigma -
\frac{1}{2sW''(0)}\sum_{i\neq i_0}\frac{\overline x_{i_0}(t)-
\overline x_i(t)}{|\overline x_{i_0}(t)-\overline x_i(t)|^{1+2s}}
& = & \frac{\sigma(t,x)-\sigma(t,\overline x_{i_0}(t))}{W''(0)}\\
& = & O(x-\overline x_{i_0}(t))\\
&= & O(\epsilon^{\gamma}),
\end{eqnarray*}
thanks to~\eqref{sigma} and~\eqref{caso1}.
Hence~\eqref{777} reads 
\begin{equation}\label{888}
I_{\epsilon}= e^{i_0}_{\epsilon}+\beta\tilde\sigma-\sigma + O(\epsilon^{\gamma})+ O(\epsilon^{\vartheta-2s}) +O(\epsilon^{\gamma}).
\end{equation}
Also, in the light of \eqref{tildesigma}, we see that
\begin{equation}\label{8.43bis}
\beta\tilde\sigma-\sigma=\delta>0.\end{equation}
Now, we claim that  
\begin{equation}\label{err_zero}
{\mbox{ the error~$e^{i_0}_{\epsilon}$ (that was defined 
in~\eqref{ei})
tends to zero as $\epsilon\rightarrow0$.}}
\end{equation} 
For this, we notice that~$\psi_i=\psi\left(\frac{x-\overline x_i(t)}{\epsilon}\right)$, 
with~$i\ne i_0$, tends to zero because 
of the behavior of
the corrector at infinity
(recall~\eqref{unif} and~\eqref{varteta}).
Moreover, thanks to~\eqref{phi} and~\eqref{varteta}
we have that, 
for~$i\ne i_0$,
$$ \tilde u_i=u\left(\frac{x-\overline x_i(t)}{\epsilon}\right)-H\left(\frac{x-\overline x_i(t)}{\epsilon}\right)= O\left(\frac{\epsilon^{2s}}{|x-\overline x_i(t)|^{2s}}\right)=O(\epsilon^{2s}) $$
and 
$$ \frac{(\tilde u_i)^2}{\epsilon^{2s}} =\frac{O(\epsilon^{4s})}{\epsilon^{2s}}=O(\epsilon^{2s}), $$
thus proving~\eqref{err_zero}. 

Hence, from~\eqref{888}, \eqref{8.43bis} and~\eqref{err_zero} 
we obtain that for $\epsilon$ sufficiently small 
$$ I_{\epsilon}\geq\frac{\delta}{2}>0, $$
which implies~\eqref{GP} in this case.

\noindent
\\ {\it Case~2:} Suppose that $|x-\overline 
x_i(t)|>\epsilon^{\gamma}$ for every 
$i\in\left\lbrace1,\ldots,N\right\rbrace$.
In this case, we can fix~$i_0$ arbitrarily, say~$i_0:=1$
for concreteness. We use~\eqref{999} to obtain
\begin{eqnarray*}
\biggl|\sum_{i\neq i_0}\left(\frac{\tilde u_i}{\epsilon^{2s}}+\frac{1}{2sW''(0)}\frac{x-\overline x_i(t)}{|x-\overline x_i(t)|^{1+2s}}\right)\biggr| 
&\leq& \frac{C\, \epsilon^{\vartheta}}{\epsilon^{2s}}\sum_{i\neq i_0}\frac{1}{|x-\overline x_i(t)|^{\vartheta}}\\
&\leq& C\frac{\epsilon^{\vartheta-2s}}{\epsilon^{\gamma\vartheta}}  =
C\, \epsilon^{\vartheta-2s-\gamma\vartheta}. 
\end{eqnarray*}
Therefore, by formula~\eqref{10.6} and the definition of $\tilde\sigma$ in 
\eqref{tildesigma} we have 
\begin{equation}\label{1140}
I_{\epsilon}= e^{i_0}_{\epsilon}+\delta+O(\tilde u_{i_0})
\biggl(\eta\, \overline c_{i_0}+
\tilde\sigma -\frac{1}{2sW''(0)}\sum_{i\neq i_0}
\frac{x-\overline x_i(t)}{|x-\overline x_i(t)|^{1+2s}} \biggr)+ 
O(\epsilon^{\vartheta-2s-\gamma\vartheta}). 
\end{equation}
Now we observe that, for any~$i\ne i_0$,
\begin{equation}\label{1139}
\left| \frac{x-\overline 
x_i(t)}{|x-\overline x_i(t)|^{1+2s}}\right|
\le \frac{1}{|x-\overline x_i(t)|^{2s}}\le 
\frac{1}{\epsilon^{2\gamma s}}=O(\epsilon^{-2\gamma s}).\end{equation}
Notice that this term is divergent as $\epsilon$ tends to
zero. Therefore, from~\eqref{1139} we conclude that
$$ \eta\, \overline c_{i_0}+
\tilde\sigma -\frac{1}{2sW''(0)}\sum_{i\neq i_0}
\frac{x-\overline x_i(t)}{|x-\overline x_i(t)|^{1+2s}}=
O(\epsilon^{-2\gamma s}),$$
since the other terms are bounded. By plugging this
into~\eqref{1140} we obtain 
\begin{equation}\label{1141}
I_{\epsilon}= e^{i_0}_{\epsilon}+\delta+O(\tilde u_{i_0})\cdot
O(\epsilon^{-2\gamma s})
+O(\epsilon^{\vartheta-2s-\gamma\vartheta}).
\end{equation}
Now we observe that
for every $i\in\left\lbrace1,\ldots,N\right\rbrace$, 
\begin{equation}\label{9112}\begin{split}
\tilde u_{i}&=u\left(\frac{x-\overline 
x_i(t)}{\epsilon}\right)\!-H\left(\frac{x-\overline x_i(t)}{\epsilon}\right)\\
&=O\left(\frac{\epsilon^{2s}}{|x-\overline x_i(t)|^{2s}}\right) 
=O\left(\frac{\epsilon^{2s}}{
\epsilon^{2\gamma s}}\right)
=O\left(\epsilon^{2s(1-\gamma)}\right).
\end{split}\end{equation}
As a consequence
\begin{equation}\label{9113}
\frac{(\tilde u_i)^2}{\epsilon^{2s}}=O\left(\epsilon^{2s(1-2\gamma)}\right) \quad
{\mbox{ and }} \quad
O(\tilde u_{i_0})\cdot
O(\epsilon^{-2\gamma s})=O\left(\epsilon^{2s(1-2\gamma)}\right). 
\end{equation}
We observe that, since~$\vartheta\le 4s$ (see \eqref{SGAMMA} and recall 
that $\alpha=4s$), from \eqref{scelta} we have 
\begin{equation}\label{maggiore2}
1-2\gamma>1-\frac{2(\vartheta-2s)}{\vartheta}= \frac{4s-\vartheta}{\vartheta}\ge 0. \end{equation}
Also, notice that, thanks again to~\eqref{scelta}, 
\begin{equation}\label{maggiore3}
\vartheta-2s-\gamma\vartheta>0.
\end{equation} 
By inserting~\eqref{9113} into~\eqref{1141} 
and recalling~\eqref{maggiore2} and~\eqref{maggiore3} we get
\begin{equation}\label{1146}
I_{\epsilon}= e^{i_0}_{\epsilon}+\delta+O(\epsilon^{\alpha}),
\end{equation}
for some~$\alpha>0$.
Now we check that 
\begin{equation}\label{9114}
{\mbox{the error term $e^{i_0}_{\epsilon}$ tends to zero as 
$\epsilon\rightarrow0$.}}\end{equation}
For this, we remark that, in 
this case,
$$ \frac{|x-\overline x_i(t)|}{\epsilon}\ge
\frac{\epsilon^{\gamma}}{\epsilon}=\epsilon^{\gamma-1},$$
which diverges for small~$\epsilon$, since~$\gamma<1$. 
Therefore, for $x$ fixed
as in the assumption of Case~2, we have that
$$ \psi_i(x)=\psi\left( 
\frac{x-\overline x_i(t)}{\epsilon}\right)\longrightarrow 0$$
as~$\epsilon\rightarrow 0$, due to the infinitesimal behavior
of~$\psi$ at infinity (see~\eqref{unif}).
Using this,
\eqref{9112}, \eqref{9113} and the definition of
the error term given in~\eqref{ei}, we obtain~\eqref{9114}.

Hence, by using~\eqref{9114} inside~\eqref{1146}
and recalling that~$\delta>0$, we conclude that
$$ I_{\epsilon}\geq\frac{\delta}{2}>0$$
for $\epsilon$ sufficiently smooth, thus proving~\eqref{GP}
in this case too.
\end{proof}

\end{document}